\documentclass[11pt,a4paper]{article}

\usepackage[english]{babel}
\usepackage{amsmath}
\usepackage{amsfonts}
\usepackage{amsthm}
\usepackage{amscd}
\usepackage{amssymb}
\usepackage{bm}
\usepackage[utf8]{inputenc}
\usepackage{textcomp}
\usepackage{enumerate}
\usepackage{graphicx}
\usepackage[margin=3.1cm]{geometry}

\usepackage{tikz}

\newtheoremstyle{mystyle}{}{}{\slshape}{2pt}{\scshape}{.}{ }{}
\newtheoremstyle{etapestyle}{}{}{\itshape}{2em}{\sffamily}{:}{ }{\thmname{#1}}
\newtheoremstyle{definitionstyle}{}{}{}{2pt}{\bfseries}{.}{ }{}

\newtheorem{thm}{Theorem}[section]

\newtheorem{prop}[thm]{Proposition}
\newtheorem{defi}[thm]{Definition}

\newtheorem{proposition}[thm]{Proposition}
\newtheorem{propdef}[thm]{Proposition-Definition}

\theoremstyle{mystyle}

 \theoremstyle{remark}

\theoremstyle{etapestyle}

\theoremstyle{definitionstyle}

\newcommand{\op}{\operatorname}
\newcommand{\rw}{\rightarrow}

\newcommand{\mav}{\mathcal{V}}

\newcommand{\maf}{\mathcal{F}}
\newcommand{\mai}{\mathcal{I}}
\newcommand{\mah}{\mathcal{H}}

\newcommand{\mao}{\mathcal{O}}

\newcommand{\maz}{\mathcal{Z}}
\newcommand{\mau}{\mathcal{U}}

\newcommand{\gr}{\mathcal{G}}

\newcommand{\taut}{\mathcal{T}}

\newcommand{\si}{\sigma}
\newcommand{\enumr}{\begin{enumerate}[(R1)]\addtolength{\itemsep}{0.8\baselineskip}}
\newcommand{\enumrr}{\begin{enumerate}[(SR2)]\addtolength{\itemsep}{0.8\baselineskip}}
\newcommand{\enuml}{\begin{enumerate}[(L1)]\addtolength{\itemsep}{0.8\baselineskip}}
\newcommand{\enumc}{\begin{enumerate}[(C1)]\addtolength{\itemsep}{0.8\baselineskip}}
\newcommand{\enumM}{\begin{enumerate}[(M1)]\addtolength{\itemsep}{0.8\baselineskip}}
\newcommand{\enumrrr}{\begin{enumerate}[(R3)]\addtolength{\itemsep}{0.8\baselineskip}}

\newcommand{\lo}{\operatorname{log}}

\newcommand{\txd}{T_X(-\lo D)}

\newcommand{\txdd}{T_{X'}(-\lo D')}

\newcommand{\lieg}{\mathfrak{g}}
\newcommand{\mk}{\mathfrak}
\newcommand{\mab}{\mathcal{B}}

\begin{document}

\title{Log homogeneous compactifications of some classical groups}

\author{Mathieu Huruguen}
\date{}
\maketitle

\begin{abstract}
We generalize in positive characteristics some results of Bien and Brion on log homogeneous compactifications of 
a homogeneous space under the action of a connected reductive group. 
We also construct an explicit smooth log homogeneous compactification of the general linear group by successive blow-ups starting from a grassmannian. 
By taking fixed points of certain involutions on this compactification, we obtain
smooth log homogeneous compactifications of the special orthogonal and the symplectic groups.

\end{abstract}

\section*{Introduction}

Let $k$ be an algebraically closed field and $G$ a connected reductive group defined over $k$. 
Given a homogeneous space $\Omega$ under the action of the group $G$ it is natural to consider equivariant 
\textbf{compactifications} or partial equivariant compactifications 
of it. \textbf{Embeddings} are normal irreducible varieties equipped with an action of $G$ and containing $\Omega$ as a dense orbit, and compactifications 
are complete embeddings. 
Compactifications have shown to be powerful tools to produce interesting 
representations of the group $G$ or to solve enumerative problems. In the influent paper \cite{Lu}, Luna and Vust developed a classification theory of embeddings of the 
homogeneous space $\Omega$ assuming that the field $k$ is of characteristic zero. Their theory can be made very explicit and extended to all characteristics, see
for instance \cite{Kn}, in the 
\textbf{spherical} case, 
that is, when a Borel subgroup of $G$ possesses a dense orbit in the homogeneous space $\Omega$. In this case, the embeddings of $\Omega$ are classified 
by combinatorial objects called \textbf{colored fans}. If the homogeneous space is a torus acting on itself by multiplication then one recovers the classification 
of torus embeddings or toric varieties in terms of fans, see for instance \cite{KKMSD}.

\medskip

In the first part of the paper we focus on a certain category of ``good'' compactifications of the homogeneous space $\Omega$. For example, these compactifications
are smooth and the boundaries are strict normal crossing divisors. There are several notions of ``good'' compactifications in the literature. Some of them are defined by 
geometric conditions, as for example the {\bf toroidal} compactifications of Mumford \cite{KKMSD}, the {\bf regular} compactifications of Bifet De Concini and Procesi \cite{Bif}, 
the {\bf log homogeneous} compactifications of Brion \cite{Br} and some of them are defined by conditions from the embedding theory of Luna and Vust, as for example 
the {\bf colorless} compactifications. As it was shown by Bien and Brion \cite{Br}, if the base field $k$ is of characteristic zero then the homogeneous space $\Omega$ 
admits 
a log homogeneous compactification if 
and only if it is spherical, and in that case the four different notions of ``good'' compactifications mentioned above coincide.  
We generalize their results in positive characteristics in Section \ref{first}. We prove that a homogeneous space admitting a log homogeneous compactification 
is necessarily {\bf separably spherical} in the sense of Proposition-Definition \ref{conditionb}. In that case, we relate the log homogeneous compactifications to 
the regular and the 
colorless one, see Theorem \ref{affaiblissement} for a precise statement. 
We do not know whether the condition of being separably spherical is sufficient
for a homogeneous space to have a 
log homogeneous compactification. Along the way we prove Theorem \ref{structurelocale}, which is of independent interest, on the local structure of 
colorless compactifications of
spherical homogeneous 
spaces, generalizing a result of Brion, Luna and Vust, see \cite{Br4}. 

\medskip
In Section \ref{explicit}, we focus on the explicit construction of equivariant compactifications of a connected reductive group. That is, the homogeneous space $\Omega$ is 
a connected reductive group $G$ acted upon by $G\times G$ by left and right translations. The construction of ``good'' compactifications
of a reductive group is a very old problem, with roots in the 19th century in the work of 
Chasles, Schubert, who were motivated by questions from enumerative geometry. 
When the group $G$ is semi-simple there is a particular compactification 
$\overline{G}$ called $\textbf{canonical}$ which possesses interesting properties, making it particularly convenient to work with. 
For example, the boundary is a divisor whose irreducible component intersect properly
and the closure of the $G\times G$-orbits are exactly the partial intersections of these prime divisors. Also, there is a unique closed orbit of $G\times G$ in the 
canonical compactification of $G$. Moreover, every toroidal compactification of $G$ has a dominant equivariant morphism to $\overline{G}$. 
If the canonical compactification $\overline{G}$ is smooth, then it is $\textbf{wonderful}$ in sense of Luna \cite{Lu1}. 
When the group $G$ is of adjoint type, its canonical compactification is smooth, and there are many known constructions of this 
wonderful compactification, see for example \cite{Ty}, \cite{La1}, \cite{La2}, \cite{La3}, \cite{Va}, \cite{ThKl}, \cite{Th}
for the case of the projective linear group $\op{PGL}(n)$ and
\cite{DCP}, \cite{St}, \cite{Br3} for the general case. In general the canonical compactification is not smooth, as it can be seen for example when $G$ is
the special orthogonal group $\op{SO}(2n)$.

\medskip

One way to construct a compactification 
of $G$ is by considering a linear representation $V$ of $G$ and taking the closure of $G$ in the projective space $\mathbb{P}(\op{End}(V))$. The compactifications
arising in this way are called linear. It was shown by De Concini and Procesi \cite{DCP} that the linear compactifications of a semi-simple group of adjoint type 
are of particular interest. Recently, Timashev \cite{Tim}, Gandini and Ruzzi \cite{GanRuz}, found combinatorial criterions for certain linear compactifications to be normal, 
or smooth. In \cite{Gan}, Gandini classifies the linear compactifications of the odd special orthogonal group having one closed orbit. 
By a very new and elegant approach, Martens and Thaddeus \cite{MaTh} recently discovered a general construction of the toroidal compactifications of 
a connected reductive group $G$ 
as the coarse moduli spaces of certain algebraic stacks parametrizing objects called ``framed principal $G$-bundles over chain of lines''.

\medskip

Our approach is much more classical. In Section \ref{explicit}, we construct a log homogeneous compactification $\gr_n$ of the general linear group $\op{GL}(n)$ 
by successive blow-ups, starting from 
a grassmannian. The compactification $\gr_n$ is defined over an arbitrary base scheme. We then identify the compactifications of the special orthogonal group or 
the symplectic group obtained by taking the 
fixed points of certain involutions on the compactification $\gr_n$. This provides a new construction of the wonderful compactification of the odd orthogonal group 
$\op{SO}(2n+1)$, which is of adjoint type, of the symplectic group $\op{Sp}(2n)$, which is not of adjoint type, and of a toroidal desingularization 
of the canonical compactification of the even orthogonal group $\op{SO}(2n)$ having only two closed orbits. This is the minimal number of closed orbits on a smooth 
log homogeneous compactification, as the canonical compactification of $\op{SO}(2n)$ is not smooth.

\medskip

Our procedure is similar to that used by Vainsencher, see \cite{Va}, to construct the wonderful compactification of the projective linear group $\op{PGL}(n)$ or that 
of Kausz, see \cite{Ka}, to construct his compactification of the general linear group $\op{GL}(n)$. However, unlike Kausz, we are not able to describe the functor of points of 
our compactification $\gr_n$. In that direction, we obtained a partial result in \cite{Huphd}, where we describe the set $\gr_n(K)$ for every field $K$. We decided not 
to include this description in the present paper, as it is long and technical.
The functor of points of the wonderful compactification of the projective linear group is described in \cite{ThKl} and that of the symplectic group is described in 
\cite{Ab}.

\section*{Acknowledgments}
This paper is part of my Ph.D thesis. I would like to thank my supervisor Michel Brion for his precious advices and his careful reading. I would also like to thank Antoine 
Chambert-Loir and Philippe Gille for helpful comments on a previous version of this paper.

\section{Log homogeneous compactifications}\label{first}
First we fix some notations.
Let $k$ be an algebraically closed field of arbitrary characteristic $p$. 
By a variety over $k$ we mean a separated integral $k$-scheme of finite type. If $X$ is a variety over $k$ and $x$ is a point of $X$, we denote by $T_{X,x}$ the tangent space of 
$X$ at $x$. If $Y$ is a subvariety of $X$ containing $x$, we denote by $N_{Y/X,x}$ the normal space to $Y$ in $X$ at $x$. 

\medskip

For an algebraic group $G,H,P\ldots$ we denote by the corresponding gothic letter $\mathfrak{g},\mathfrak{h},\mathfrak{p}\ldots$
its Lie algebra. Let $G$ be a connected reductive group defined over $k$. A $G$-variety is a variety equipped with an action of $G$. 
Let $X$ be a $G$-variety. For each point $x\in X$ we denote by $G_x$ the isotropy group scheme of $x$. We also denote by $orb_x$ the morphism
$$orb_x : G\rw X,\quad g\mapsto g\cdot x.$$
The orbit of $x$ under the action of $G$ is called separable if the morphism $orb_x$ is, that is, if its differential is surjective, or, equivalently, if the
group scheme $G_x$ is reduced.

\medskip

We fix a homogeneous space $\Omega$ under the action of $G$. Let $X$ be a smooth compactification of $\Omega$, that is, a complete smooth $G$-variety containing $\Omega$ as an open dense orbit. We suppose that the 
complement $D$ of $\Omega$ in $X$ is a strict normal crossing divisor. 
\medskip

In \cite{Bif}, Bifet, De Concini and Procesi introduce and study the regular compactifications of a homogeneous space over an algebraically closed field 
of characteristic zero.
We generalize their definition in two different ways :

\begin{defi}\label{defreg}
The compactification $X$ is $\textbf{regular}$ (resp. $\textbf{strongly regular}$) if the orbits of $G$ in $X$ are separable, the partial intersections of the 
irreducible components of $D$ are 
precisely the closures of the $G$-orbits in $X$ and, for each point $x\in X$, the isotropy group $G_x$ possesses an open (resp. open and separable) 
orbit in the normal space $N_{Gx/X,x}$ to the orbit $Gx$ in $X$ at the point $x$. 
\end{defi}
If the characteristic of the base field $k$ is zero, then the notion of regular and strongly regular coincide with the original notion of \cite{Bif}. 
This is no longer true in positive characteristic, as we shall see at the end of Section \ref{comparison}.

\medskip

In \cite{Br}, Brion defines the log homogeneous compactifications over an algebraically closed field - throughout his paper the base field is also of characteristic 
zero, but the definition makes sense in arbitrary characteristic.
Recall that the logarithmic tangent bundle $T_X(-\op{log} D)$ is the vector bundle over $X$ whose sheaf of section is the subsheaf of the tangent sheaf of $X$ consisting of 
the derivations that preserve the ideal sheaf $\mathcal{O}_X(-D)$ of $D$. As $G$ acts on $X$ and $D$ is stable under the action of $G$, it is easily seen that 
the infinitesimal action of the Lie algebra $\lieg$ on $X$ gives rise to a natural vector bundle morphism:
$$X\times\lieg\rightarrow \txd.$$
We refer the reader to \cite{Br} for further details.
 \begin{defi}\label{defilh}
The compactification $X$ is called $\textbf{log homogeneous}$ if the morphism of vector bundles on $X$:
$$X\times\lieg\rightarrow \txd$$
is surjective.
\end{defi} 
 
\medskip
Assuming that the characteristic of the base field is zero, Bien and Brion prove in \cite{Bi} that the homogeneous space $\Omega$ possesses a log homogeneous 
compactification if and only if it is spherical. In this case, they also prove that it is equivalent for a smooth compactification $X$ of $\Omega$ to be log homogeneous, regular 
or to have no color - as an embedding of a spherical homogeneous space, see \cite{Kn}.
Their proof relies heavily on a local structure theorem for spherical varieties in characteristic 
zero established by Brion, Luna and Vust in \cite{Br4}. 

\medskip

A generalization of the local structure theorem was obtained by Knop in \cite{Kn1}; essentially, one has to replace in the statement of that theorem an isomorphism
by a finite surjective morphism. In Section \ref{gen}
we shall prove that under a separability assumption, the finite surjective morphism in Knop's theorem is an isomorphism.
Then, in Section \ref{comparison} we prove that the smooth compactification $X$ of $\Omega$ is regular 
if and only if the homogeneous space $\Omega$ is spherical, the embedding $X$ has no color and each closed orbit of $G$ in $X$ is separable (Theorem \ref{sscsr}).
We also prove that the smooth compactification $X$ of $\Omega$ is strongly regular if and only if it is log homogeneous (Theorem \ref{lhfr}). Finally, we exhibit a class 
of spherical homogeneous spaces for which the notion of regular and strongly regular compactifications coincide.
In Section \ref{fixed} we show that log homogeneity is preserved under taking fixed points by an automorphism
of finite order prime to the characteristic of the base field $k$.
In Section \ref{example} we recall the classification of Luna and Vust in the setting of compactification of reductive groups, as this will be useful in Section \ref{explicit}.

\subsection{A local structure theorem}\label{gen}

Let $X$ be a smooth $G$-variety. We assume that there is a unique closed orbit $\omega$ of $G$ in $X$ and that this orbit 
is complete and separable. We fix a point $x$ on $\omega$. The isotropy group $G_x$ is a parabolic subgroup of $G$. We fix a Borel subgroup $B$ of $G$ such that $BG_x$ 
is open in $G$. We fix a maximal torus $T$ of $G$ contained in $G_x$ and $B$ and we denote by $P$ the opposite parabolic subgroup to $G_x$ containing $B$.
We also denote by $L$ the Levi subgroup of $P$ containing $T$ and by $R_u(P)$ the unipotent radical of $P$.
With these notations we have the following proposition, which relies on a result of Knop \cite[Theorem $1.2$]{Kn1}.
\begin{proposition}\label{key1}
There exists an affine open subvariety $X_s$ of $X$ which is stable under the action of $P$ and a closed subvariety $Z$ of $X_s$ stable under the action of $T$, 
containing $x$ such that:
\begin{enumerate}[(1)]
\item The variety $Z$ is smooth at $x$ and the vector space $T_{Z,x}$ endowed with the action of $T$ is isomorphic to the vector space $N_{\omega/X,x}$ endowed with 
the action of $T$.
\item The morphism:
$$\mu: R_u(P)\times Z\rightarrow X_s,\quad (p,z)\mapsto p\cdot z$$
is finite, surjective, étale at $(e,x)$, and the fiber $\mu^{-1}(x)$ is reduced to the single point $\{(e,x)\}$.
\end{enumerate}
\end{proposition}
\begin{proof}
As the smooth $G$-variety $X$ has a unique closed orbit, it is quasi-projective by a famous result of Sumihiro, see \cite{Sum}. 
We fix a very ample line bundle $\mathcal{L}$ on $X$. We fix a $G$-linearization of this line bundle.
By \cite[Theorem $2.10$]{Kn1}, there exists an integer $N$ and a global section $s$ of $\mathcal{L}^N$ 
such that the nonzero locus $X_s$ of $s$ 
is an affine open subvariety containing the point $x$ and the stabilizer of the line spanned by $s$ in the vector space $H^0(X,\mathcal{L}^N)$ is
$P$. The open subvariety $X_s$ is therefore affine, contains the point $x$ and is stable under the action of the parabolic subgroup $P$.
Using the line bundle $\mathcal{L}^N$, we embed $X$ into a projective space $\mathbb{P}(V)$ on which $G$ acts linearly.
We choose a $T$-stable complement $S$ to $T_{\omega,x}$ in the tangent space $T_{\mathbb{P}(V),x}$, such that $S$ is the direct sum 
of a $T$-stable complement of $T_{\omega,x}$ in $T_{X,x}$ and a $T$-stable complement of $T_{X,x}$ in $T_{\mathbb{P}(V),x}$.
This is possible because $T$ is a linearly reductive group.

\medskip

We consider now the linear subspace $S'$ of 
$\mathbb{P}(V)$ containing $x$ and whose tangent space at $x$ is $S$. It is a $T$-stable subvariety of $\mathbb{P}(V)$.
By \cite[Theorem $1.2$]{Kn1}, there is an irreducible component $Z$ of $X_s\cap S'$ containing $x$ and such that the morphisms
$$\mu: R_u(P)\times Z\rightarrow X_s,\quad (p,z)\mapsto p\cdot z$$
$$\nu: Z\rightarrow X_s/R_u(P),\quad z\mapsto zR_u(P)$$
are finite and surjective. Moreover, the fiber $\mu^{-1}(x)$ is reduced to the single point $(e,x)$. 
We observe now that $S'$ intersects $X_s$ transversally at
$x$. This implies that the subvariety $Z$ is smooth at $x$. It is also $T$-stable, as an irreducible component of $X_s\cap S'$. By definition, the parabolic subgroup $P$ contains the Borel subgroup $B$, therefore the orbit $Px=R_u(P)x$ 
is open in $\omega$. Moreover, we have the direct sum decomposition $\lieg=\lieg_x\oplus \mathfrak{p}_u,$
where $\mathfrak{p}_u$ is the Lie algebra of the unipotent radical $R_u(P)$ of $P$.
The morphism
$$d_e orb_x:\lieg\rightarrow T_{\omega,x}$$
is surjective and identically zero on $\lieg_x$. This proves that the restriction of this morphism to $\mathfrak{p}_u$
is an isomorphism. 
The morphism 
$$\mu:R_u(P)\times Z\rightarrow X_s,\quad (p,z)\mapsto p\cdot z$$
is therefore étale at $(e,x)$. Indeed, its differential at this point is:
$$\mathfrak{p}_u\times T_{Z,x}\rightarrow T_{X,x},\quad (h,k)\mapsto d_e orb_x(h)+k.$$ 
We also see that the spaces 
$T_{Z,x}$ and $N_{\omega/X,x}$ endowed with their action of the torus $T$ are isomorphic, completing the proof of the proposition.
\end{proof}

We now suppose further that $X$ is an embedding of the homogeneous space $\Omega$. With this additional assumption we have:
\begin{thm}\label{structurelocale}
The following three properties are equivalent:
\begin{enumerate}[(1)]
\item The homogeneous space $\Omega$ is spherical and the embedding $X$ has no color.
\item The torus $T$ possesses an open orbit in the normal space $N_{\omega/X,x}$. Moreoever, the complement $D$ of $\Omega$ in $X$ 
is a strict normal crossing divisor and the partial intersections of the irreducible components of $D$ are the closure of the $G$-orbits in $X$.
\item The set $X_0=\{y\in X,\quad x\in \overline{B y}\}$ is an affine open subvariety of $X$ which is stable by $P$. 
Moreover, there exists a closed subvariety $Z$ of $X_0$ which is smooth, stable by $L$, on which the derived subgroup $[L,L]$ acts trivially
and containing an open orbit of the torus $L/[L,L]$, such that the morphism:
$$R_u(P)\times Z\rightarrow X_0,\quad (p,z)\mapsto p\cdot z$$
is an isomorphism. Finally, each orbit of $G$ in $X$ intersects $Z$ along a unique orbit of $T$.
\end{enumerate}
\end{thm}

\begin{proof}
$(3)\Rightarrow (1)$ As $T$ possesses an open orbit in $Z$, we see that the Borel subgroup $B$ has an open orbit in $X$, and the homogeneous 
 space $\Omega$ is spherical. Moreover, let $D$ be a $B$-stable prime divisor on $X$ containing $\omega$. Using the isomorphism 
 in $(3)$ we can write
 $$D\cap X_0=R_u(P)\times (D\cap Z).$$
As $D\cap Z$ is a closed irreducible $T$-stable subvariety of $Z$, it is the closure of a $T$-orbit in $Z$. As the $T$-orbits in $Z$ correponds bijectively 
to the $G$-orbits in $X$, we see that $D$ is the closure of a $G$-orbit in $X$ and is therefore stable under the action of $G$. This proves that the embedding
$X$ of $\Omega$ has no color.

\medskip

$(3)\Rightarrow (2)$ The isomorphism in $(3)$ proves that the spaces $T_{Z,x}$ and $N_{\omega/X,x}$ endowed with their actions of the torus $T$ 
are isomorphic. As $T$ possesses an open dense orbit in the first one, it also has an open dense orbit in the latter.
As $Z$ is smooth toric variety, we see that the complement of the open orbit of $T$ in $Z$ is 
a strict normal crossing divisor whose associated strata are the $T$-orbits in $Z$. 
Using the isomorphism given by $(3)$, we see that the complement of the open orbit of the parabolic subgroup $P$ in $X_0$ is a strict normal crossing divisor 
whose associated strata are the products $R_u(P)\times \Omega'$, where $\Omega'$ runs over the set of $T$-orbits in $Z$.
To complete the proof that property $(2)$ is satisfied, we translate the open subvariety $X_0$ by elements of $G$ and we use the fact that each $G$-orbit in $X$ 
intersects $Z$ along a unique $T$-orbit.

\medskip

$(1)\Rightarrow (3)$
We use the notations of Proposition \ref{key1}. By \cite[Lemma $6.5$]{Kn} the fact that the embedding $X$ has no color implies that the parabolic subgroup $P$ 
is the stabilizer of the open $B$-orbit $\Omega$ in $X$. Using this fact and \cite[Theorem $2.8$]{Kn1} we obtain that the derived subgroup of $P$, and therefore 
the derived group of $L$, acts trivially on $X_s/R_u(P)$. Moreover, as the homogeneous space $\Omega$ is spherical, the Levi subgroup $L$ has an open orbit in 
$X_s/R_u(P)$. The torus $T$ has therefore an open orbit in $X_s/R_u(P)$, as the derived group of $L$ acts trivially. Using the finite surjective morphism 
$\nu$ appearing in the proof of Proposition \ref{key1}, we see that $T$ has an open orbit in $Z$.
$Z$ is therefore a smooth affine toric variety with a fixed point under the action of a quotient of $T$. Moreover, as the subvariety $Z$ is left stable under the action of $T$ and the derived group 
$[L,L]$ acts trivially on $Z$, we see that the Levi subgroup $L$ leaves the subvariety $Z$ invariant.

\medskip

We observe now that the locus of points of $R_u(P)\times Z$ where $\mu$ is not étale is closed and stable under the actions of 
$R_u(P)$ and $T$. The unique closed orbit of 
$R_u(P)\rtimes T$ in $R_u(P)\times Z$ is $R_u(P)x$ and $\mu$ is étale at $(e,x)$, therefore we obtain that $\mu$ is an étale morphism.
As the morphism $\mu$ is also finite of degree $1$ (the fiber of $\{x\}$ being reduced to a single point), it is an isomorphism.

\medskip

We prove now that each $G$-orbit in $X$ intersects $Z$ along a unique $T$-orbit.
First, we observe that, as $\omega$ is the unique closed orbit of $G$ in $X$, the open subvariety $X_s$ intersects every $G$-orbit.
We shall prove that the closures of the $G$-orbits in $X$ corresponds bijectively to the closures of the $T$-orbits in $Z$.
Let $X'$ be the closure of a $G$-orbit in $X$. As $X'$ is the closure of $X'\cap X_s$, it is also equal, using the isomorphism $\mu$, 
to the closure of $R_u(P)(X'\cap Z)$. The closed subvariety  $X'\cap Z$ of $Z$ is therefore a closed irreducible $T$-stable
subvariety. We can conclude that it is the closure of a $T$-orbit in $Z$. 
Conversely let $Z'$ be the closure of a $T$-orbit in $Z$. As $Z$ is a smooth toric variety, we can write   
$$Z'=D'_1\cap
D'_2\cap...\cap D'_r,$$ where the $D'_i$s are $T$-stable prime divisors on $Z$. We observe that the primes divisors
$$\overline{R_u(P)D'_1},...,\overline{R_u(P)D'_r}$$ on $X$ are stable under the action of $P$. Indeed, the orbits of $P$ in $X_s$ are exactly the orbits 
of
$R_u(P)\rtimes T$ in $X_s$. As $X$ has no color, the fact that these divisors contain the closed orbit $\omega$ proves that they are 
stable under the action of $G$. Their intersection $\overline{R_u(P)Z}$ is also $G$-stable. As it is irreducible, we can conclude that it is the closure of a $G$-orbit in $X$.

\medskip

In order to complete the proof that $(1)\Rightarrow (3)$, it remains to show that 
$$X_s=\{y\in X,\quad x\in \overline{B y}\}.$$
Let $y$ be a point on $X$ such that $x$ belongs to $\overline{B y}$. The intersection $X_s\cap
\overline{B y}$ is a non empty open subset of $\overline{B y}$ which is stable under the action of $B$.
Therefore it contains $y$, that is, $y$ belongs to $X_s$.
Now let $y$ be a point on $X_s$. The closed subvariety $\overline{B y}$ contains a closed $B$-orbit in $X_s$. As the unique closed orbit 
of $B$ in $X_s$ is the orbit of $x$, we see that $x$ belongs to $\overline{B y}$, completing the argument.

\medskip

($2\Rightarrow 3$) We use the notations introduced in Proposition \ref{key1}.  
By assumption, the torus $T$ possesses an open orbit in the normal space $N_{\omega/X,x}$. Moreover, by Proposition \ref{key1}, the spaces 
 $T_{Z,x}$ and $N_{\omega/X,x}$ endowed with their actions of $T$ are isomorphic. 
 Therefore, the torus $T$ possesses an open dense orbit in $T_{Z,x}$. It is then an easy exercise left to the reader to prove that 
 the variety $Z$ is a smooth toric variety for a quotient of $T$.
The same arguments as above prove that the morphism $\mu$ is an isomorphism.

\medskip
We prove now that each $G$-orbit in $X$ intersects $Z$ along a unique orbit of $T$. Let $D$ be the complement of $\Omega$ in $X$. By assumption, 
it is a strict normal crossing divisor whose associated strata are the $G$-orbits in $X$. 
We denote by $D_1,\ldots,D_r$ the irreducible component of $D$. As there is a unique closed orbit of $G$ on $X$ each partial intersection 
$\bigcap_{i\in I}D_i$ is non empty and irreducible or, in other words, it is a stratum of $D$. The integer $r$ is the codimension of 
the closed orbit $\omega$ in $X$, and there are exactly $2^r$ $G$-orbits in $X$. As the variety $Z$ is a smooth affine toric variety of dimension $r$ with a fixed point, 
we see that there are exactly $2^r$ orbits of $T$ on $Z$. As each orbit of $G$ in $X$ intersect $Z$ we see that the intersection of a $G$-orbit 
with $Z$ is a single $T$-orbit.

\medskip

Finally, we prove that the open subvariety $X_s$ is equal to $X_0$ by the same argument as in the proof of $(1)\Rightarrow (3)$, completing the proof of the theorem. 
\end{proof}

\subsection{Regular, strongly regular and log homogeneous compactifications}\label{comparison}
In this section we use the following notation. Let $X$ be a $G$-variety with a finite number of orbits (for example, a spherical variety).
Let $\omega$ be an orbit of $G$ in $X$. We denote by 
$$X_{\omega,G}=\{y\in X, \omega\subseteq \overline{Gy}\}.$$
It is an open $G$-stable subvariety of $X$ in which $\omega$ is the unique closed orbit.

\begin{thm}\label{sscsr}
Let $X$ be a smooth compactification of the homogeneous space $\Omega$. The following two properties are equivalent:
\begin{enumerate}[(1)]
\item $X$ is regular.
\item The homogeneous space $\Omega$ is spherical, the embedding $X$ has no color and the orbits of $G$ in $X$ are separable.
\end{enumerate}
\end{thm}
\begin{proof}
Suppose that $X$ is regular. Let $D$ be the complement of $\Omega$ in $X$. It is a strict normal crossing divisor.
Let $\omega$ be a closed, and therefore complete and separable, orbit of $G$ in $X$. We use the notations introduced at the beginning of Section 
\ref{structurelocale} with $X_{\omega,G}$ in place of $X$. The normal space 
$N_{\omega/X,x}$ is the normal space to a stratum of the divisor $D$ and therefore possesses a natural direct sum decomposition into a sum of lines, each of them 
being stable under the action of $G_x$ (which is connected, as it is a parabolic subgroup of $G$). Therefore the representation of $G_x$ in $N_{\omega/X,x}$ factors through the action of 
a torus.This proves that the derived group of $L$ acts trivially in this space, proving that the torus $T$ has a dense orbit  
in $N_{\omega/X,x}$. By Theorem \ref{structurelocale} (applied to $X_{\omega,G}$) the homogeneous space $\Omega$ is spherical and the embedding $X_{\omega,G}$ has no color.
As this is true for each closed orbit $\omega$ of $G$ in $X$, we see that the embedding $X$ has no color.

\medskip

We assume now that $\Omega$ is spherical, $X$ has no color and that each orbit of $G$ in $X$ is separable.
By applying Theorem \ref{structurelocale} to each open subvariety $X_{\omega,G}$, where $\omega$ runs over the set of closed orbits of $X$,
we see that the complement $D$ of $\Omega$ in $X$ 
is a strict normal crossing divisor and that, for each point $x$ in $X$, the isotropy group $G_x$ has an open orbit in the normal space $N_{Gx/X,x}$.
Moreover, by assumption, the $G$-orbits in $X$ are separable. 
To complete the proof of the theorem, it remains to show that the partial intersections of the irreducible components of $D$ are irreducible.
But this is true on every colorless embedding of a spherical homogeneous space, due to the combinatorial description of these embeddings, see \cite[Section 3]{Kn}.
\end{proof}

\begin{thm}\label{lhfr}
Let $X$ be a smooth compactification of $\Omega$. The following two properties are equivalent:
\begin{enumerate}[(1)]
\item $X$ is a log homogeneous compactification.
\item $X$ is strongly regular.
\end{enumerate}
\end{thm}
\begin{proof}
We suppose first that the compactification $X$ is log homogeneous. We denote by $D$ the complement of $\Omega$ in $X$.
It is a strict normal crossing divisor. Following the argument given in  \cite[Proposition $2.1.2$]{Br} we prove that each stratum of the 
strict normal crossing divisor $D$ is a single orbit under the action of $G$ which is separable and that for each point $x\in X$, the isotropy
group $G_x$ possesses an open and separable orbit in the normal space $N_{Gx/X,x}$.
In order to conclude, it remains to prove that the partial intersection of the irreducible components of $D$ are irreducible.
But the same argument as in the proof of Theorem \ref{sscsr} prove that $\Omega$ is spherical and $X$ has no color, which is sufficient to complete the proof.

\medskip

Conversely, if $X$ is supposed to be strongly regular, the proof of \cite[Proposition $2.1.2$]{Br} 
adapts without change and shows that $X$ is a log homogeneous compactification of $\Omega$.

\end{proof}

\begin{propdef}\label{conditionb}
If the homogeneous space $\Omega$ possesses a log homogeneous compactification, then it satisfies the following equivalent conditions:
\begin{enumerate}[(1)]
\item  The homogeneous space $\Omega$ is spherical and there exists a Borel subgroup of $G$ whose open orbit in $\Omega$ is separable. 
\item  The homogeneous space $\Omega$ is spherical and the open orbit of each Borel subgroup of $G$ in $\Omega$ is separable.  
\item  The homogeneous space $\Omega$ is separable under the action of $G$, and there exists a point $x$ in $X$ and a Borel subgroup $B$ of $G$ such that  
: $\mathfrak{b}+\lieg_x=\lieg.$
 \end{enumerate}
A homogeneous space satisfying one of these properties is said to be $\textbf{separably}$ $\textbf{spherical}$.  
\end{propdef}
\begin{proof}
We suppose first that the homogeneous space $\Omega$ possesses a log homogeneous compactification $X$ and we prove that it satisfies the first condition.
By Theorem \ref{lhfr} and \ref{sscsr}, the homogeneous space $\Omega$ is spherical. Let $\omega$ be a closed, and therefore complete and separable, orbit of $G$ in $X$.
We apply Theorem \ref{structurelocale} to the open subvariety $X_{\omega,G}$. We use the notations introduced for this theorem. As $X$ is strongly regular, the maximal 
torus $T$ has an open and separable orbit in $T_{Z,x}=N_{\omega/X,x}$. As this space endowed with its action of $T$ is isomorphic to $Z$ endowed with its action of $T$, 
because $Z$ is an affine smooth toric variety with fixed point for a quotient of $T$, we see that the open orbit of $T$ in $Z$ is separable. 
Consequently, the open orbit of $R_u(P)\rtimes T$ in
$R_u(P)\times Z$ is separable, and the open orbit of $B$ in $\Omega$ is separable.

\medskip

We prove now that the three conditions in the statement of the proposition-definition are equivalent.
As the Borel subgroups of $G$ are conjugated, condition $(1)$ and $(2)$ are equivalent.
Suppose now that condition $(1)$ is satisfied. Let $B$ be a Borel subgroup of $G$ and $x$ a point in the open and separable orbit of $B$ in $\Omega$.
The linear map
$d_e orb_x :\mk{b}\rw T_{Bx,x}$
is surjective. As the orbit $Bx$ is open in $\Omega$ we see that the homogeneous space $\Omega$ is separable under the action of $G$ and that 
$$\mk{b}+\lieg_x=\lieg.$$
Conversely, we suppose that condition $(3)$ is satisfied.
As the homogeneous space $\Omega$ is separable, the linear map
$$d_e orb_x:\lieg\rw \lieg/\mathfrak{g}_x$$
is the natural projection. As we have 
$\mk{b}+\lieg_x=\lieg,$
we see that the linear map
$$d_e orb_x:\mk{b}\rw \lieg/\mathfrak{g}_x$$
is surjective. This means precisely that the orbit $Bx$ is open in $\Omega$ and separable.
\end{proof}

Here are some example of separably spherical homogeneous spaces: separable quotients of tori, partial flag varieties, symmetric spaces in characteristic not $2$ (Vust 
proves in \cite{Vu} that symmetric spaces in characteristic zero are spherical; his proof extends to characteristic not $2$ to show that symmetric spaces 
are separably spherical).

 \begin{thm}\label{affaiblissement}
We assume that the homogeneous space $\Omega$ is separably spherical. Let $X$ be a smooth compactification of $\Omega$.
The following conditions are equivalent:
\begin{enumerate}[(1)]
\item $X$ has no color and the closed orbits of $G$ in $X$ are separable.
\item $X$ is regular.
\item $X$ is strongly regular.
\item $X$ is log homogeneous under the action of $G$.
\end{enumerate}
\end{thm}

\begin{proof}
In view of Theorem \ref{sscsr} and \ref{lhfr} it suffices to show that $(1)\Rightarrow (3)$. 
We assume that condition $(1)$ is satisfied. Let $\omega$ be a closed, and therefore separable orbit of $G$ in $X$. 
We apply Theorem \ref{structurelocale} to the open subvariety $X_{\omega,G}$ of $X$ introduced in the proof of Theorem \ref{sscsr}.
We use the notations introduced for Theorem \ref{structurelocale}. As the open orbit of $B$ in $\Omega$ is separable, we see that the quotient of $T$ acting on 
$Z$ is separable. As $Z$ is a smooth affine toric variety with fixed point under this quotient, we see that the orbits of $T$ in $Z$ are all separable and 
that for each point $z\in Z$, the stabilizer $T_z$ has an open and separable orbit in the normal space $N_{Tz/Z,z}$. From this we get readily that the embedding 
$X_{\omega,G}$ of $\Omega$ satisfies the conditions defining a strongly regular embedding. As this is true for each closed orbit $\omega$, we see that $X$ is a strongly 
regular compactification of $\Omega$.
\end{proof}

We end this section with an example of a regular compactification of a homogeneous space which is not strongly regular.
We suppose that the base field $k$ has characteristic $2$. Let $G$ be the group $\op{SL}(2)$ acting on 
$X:=\mathbb{P}^1\times\mathbb{P}^1$. There are two orbits:
the open orbit $\Omega$ of pairs of distinct points and the closed orbit $\omega$, the diagonal, which has codimension one in $X$.
These orbits are separable under the action of $G$. Moreover, the complement of the open orbit, that is, the closed orbit $\omega$,
is a strict normal crossing divisor and the partial intersections 
of its irreducible components are the closure of $G$-orbits in $X$. A quick computation shows that for each point on the closed orbit $\omega$, the isotropy group 
has an open non separable orbit in the normal space to the closed orbit at that point. Therefore the compactification $X$ of $\Omega$ is regular and not strongly regular.
By Theorem \ref{affaiblissement} the homogeneous space cannot be separably spherical. This can be seen directly as follows. The homogeneous space $\Omega$ is the quotient of $G$ by a maximal torus $T$. A Borel subgroup $B$ 
of $G$ has an open orbit in $\Omega$ if and only if it does not contain $T$. But in that case the intersection $B\cap T$ is the center of $G$, which is not reduced because 
the characteristic of the base field is $2$.

\subsection{Log homogeneous compactifications and fixed points}\label{fixed}
Let $X$ be a smooth variety over the field $k$ and $\si$ an automorphism of $X$ which has finite order $r$ prime to the characteristic $p$ of $k$.
Fogarty proves in \cite{Fo} that the fixed point subscheme $X^{\si}$ is smooth and that, for each fixed point $x$ of $\si$ in $X$, the tangent space 
to $X^{\si}$ at $x$ is $T_{X,x}^{\si}$. 

\medskip

We suppose now that $X$ is a smooth log homogeneous compactification of the homogeneous space $\Omega$.
We also assume that the automorphism $\si$ leaves $\Omega$ stable and is $G$-equivariant, in the sense that there exists an automorphism $\si$ of 
the group $G$ satisfying 
$$\forall g\in G,\quad \forall x\in X,\quad \si(gx)=\si(g)\si(x).$$
By \cite[Proposition $10.1.5$]{Ri}, the neutral component $G'$ of the group $G^{\si}$ is a reductive group. 
Moreover, each connected component of the variety $\Omega^{\si}$ is a homogeneous space under the action of $G'$. We let $\Omega'$ be such a component and $X'$ be the 
connected component of $X^{\si}$ containing $\Omega'$.

\begin{prop}\label{fixedpoints}
$X'$ is a log homogeneous 
compactification of $\Omega'$ under the action of $G'$.

\end{prop}

\begin{proof}
Let $D$ be the complement of $\Omega$ in $X$. Let $D_1,\ldots,D_s$ be the irreducible components of $D$ containing $x$.
First we prove that the intersection $D':=D\cap X'$ is a strict normal crossing divisor. Let $x$ be a point in $X'$.
 For each index $i$, the intersection $D'_i:=D_i\cap X'$ is a divisor on $X'$.
Indeed, $X'$ is not contained in $D_i$ as it contains $\Omega'$. 
As $x$ is fixed by the automorphism $\si$, we can assume that the components $D_i$s are 
ordered in such a way that
$$\si(D_2)=D_1\text{ }\ldots\text{ }  \si(D_{i_1})=D_{i_1-1},\text{ } \si(D_1)=D_{i_1}$$
$$\ldots$$
$$\si(D_{i_{t-1}+2})=D_{i_{t-1}+1}\text{ }\ldots\text{ } \si(D_{i_t})=\si(D_s)=D_{i_t-1},\text{ } \si(D_{i_{t-1}+1})=D_{i_t}.$$
By convention we define $i_0=0$. For each integer $j$ from $1$ to $t$, and each integer $i$ from $i_{j-1}+1$ to $i_j$ we have $D'_i=D'_{i_j}$. 
Therefore we see that $D'_{i_j}$ is the connected component of the smooth variety $(D_{i_{j-1}+1}\cap\cdots\cap D_{i_j})^{\si}$ containing $x$. 
Consequently, it is smooth. For the moment, we have proved that $D'$ is a divisor on $X'$ whose irreducible components are smooth. 

\medskip

We prove now that the divisor $D'_{i_1},\ldots,D'_{i_t}$ intersect transversally at the point $x$.
Let $U_x$ be an open neighborhood of $x$ in $X$ which is stable by the automorphism $\si$ and on which the equation of $D$ is $u_1\ldots u_s=0$, where 
$u_1\ldots u_s\in\mathcal{O}_X(U_x)$ are part of a regular local parameter system at $x$ and satisfy:
$$\si(u_2)=u_1\text{ }\ldots\text{ }  \si(u_{i_1})=u_{i_1-1}$$
$$\ldots$$
$$\si(u_{i_{t-1}+2})=u_{i_{t-1}+1}\text{ }\ldots\text{ } \si(u_{i_t})=u_{i_t-1}.$$
We aim to prove that the images of the differential $d_x u_{i_j}$ by the natural projection 
$$(T_{X,x})^*\rw (T_{X',x})^*$$
are linearly independent, where $j$ run from $1$ to $t$.
As the point $x$ is fixed by $\si$, $\si$ acts by the differential on the tangent space $T_{X,x}$ and by the dual action on $(T_{X,x})^*$.
As the order of the automorphism $\si$ is prime to the characteristic $p$, we have a direct sum decomposition:
$$(T_{X,x})^*=((T_{X,x})^*)^{\si}\oplus\op{Ker}(id+\si+\cdots+\si^{r-1})$$
where the projection on the first factor is given by 
$$l\mapsto \frac{1}{r}(l+\si(l)+\cdots+\si^{r-1}(l)).$$
Moreover, as $T_{X',x}$ is equal to $(T_{X,x})^{\si}$, the second factor in this decomposition is easily seen to be $(T_{X',x})^{\perp}$, so that the natural 
projection 
$$(T_{X,x})^*\rw (T_{X',x})^*$$
gives an isomorphism 
$$((T_{X,x})^*)^{\si}\rw (T_{X',x})^*.$$
Finally, the images of the differential $d_x u_{i_j}$ in $(T_{X',x})^*$ are linearly independent, because the differentials $d_x u_i$ are linearly independent in 
$(T_{X,x})^*$.

\medskip

We have proved that the divisor $D'$ is a strict normal crossing divisor. We leave it as an exercise to the reader to prove that there exists a natural 
exact sequence of vector bundle on $X'$ 
$$0\rw \txdd\rw \txd_{\mid X'}\rw N_{X'/X}\rw 0,$$
and that the space $T_{X'}(-\op{log} D')_x$ is the subspace
of fixed point by $\si$ in the space $T_{X}(-\op{log} D)_x$. 
Now, the compactification $X$ is log homogeneous, therefore the linear map 
$$\lieg\rw \txd_{x}$$
is surjective. As $r$ and $p$ are relatively prime, this linear map is still surjective at the level of fixed points. That is, the linear map 
$$\lieg^{\si}\rw \txd_{x}^{\si}=T_{X'}(-\op{log} D')_x$$
is surjective. This complete the proof of the proposition.
\end{proof}

\subsection{The example of reductive groups}\label{example}
In this section the homogeneous space $\Omega$ is a connected reductive group $G$ acted upon by the group $G\times G$ by the following formula:
$$\forall (g,h)\in G\times G,\quad \forall x\in G,\quad (g,h)\cdot x=gxh^{-1}$$
We would like to explain here the classification of smooth log homogeneous compactifications of $G$. 
Observe that the homogeneous space $G$ under the action of $G\times G$ is actually separably spherical. 
By Theorem \ref{affaiblissement}, its smooth 
log homogeneous compactifications are
the smooth colorless compactifications with separable closed orbits. The last condition is actually superfluous : by \cite[Chapter $6$]{BrKu}, 
the closed orbits of $G\times G$ in a colorless 
compactification of $G$ are isomorphic to $G/B\times G/B$, where $B$ is 
a Borel subgroup of $G$. The log homogeneous 
compactifications of $G$ are therefore the smooth colorless one.

\medskip

\medskip

We now recall the combinatorial description of the smooth colorless compactifications of $G$.
Let $T$ be a maximal torus of $G$ and $B$ a Borel subgroup of $G$ containing $T$. 
We denote by $V$ the  
$\mathbb{Q}$-vector space spanned by the one-parameter subgroups of $T$ and by $\mathcal{W}$ the Weyl chamber corresponding to $B$.
Let $X$ be a smooth colorless embedding of $G$. We let the torus $T$ act ``on the left'' on $X$. For this action,  
the closure of $T$ in $X$ is a smooth complete toric variety. We associate 
to $X$ the fan consisting of those cones in the fan of the toric variety $\overline{T}$ which are included in $-\mathcal{W}$. This sets a map from 
the set of smooth colorless compactifications of $G$ to the set of fans in $V$ with support $-\mathcal{W}$ and which are smooth with respect to the lattice of one parameter subgroup in $V$.
This map is actually a bijection, see for instance \cite[Chapter $6$]{BrKu}.

\section{Explicit compactifications of classical groups}\label{explicit}
We construct a log homogeneous compactification $\gr_n$ of the general linear group $\op{GL}(n)$ by successive blow-ups, starting from 
a grassmannian. The precise procedure is explained in Section 
\ref{defofcomp}. The compactification $\gr_n$ is defined over an arbitrary base scheme. In Section \ref{sectionatlas} we study the local structure
of the action of $\op{GL}(n)\times\op{GL}(n)$ on $\gr_n$, still over an arbitrary base scheme. This enables us to compute the colored fan of $\gr_n$ over an algebraically 
closed field in Section \ref{colored}. 
Using this computation, we are able to identify the compactifications of the special orthogonal group or the symplectic group obtained by taking the 
fixed points of certain involutions on the compactification $\gr_n$. In the odd orthogonal and symplectic case we obtain the wonderful compactification.
In the even orthogonal case we obtain a log homogeneous compactification with two closed orbits. This is the minimal number of closed orbits on a smooth 
log homogeneous compactification, as the canonical compactification of $\op{SO}(2n)$ is not smooth.

\subsection{The compactifications $\gr_m$}\label{defofcomp}
As we mentioned above our construction works over an arbitrary base scheme : until the end of Section \ref{alternative} we work over a base 
scheme $S$. Let $\mav_1$ and $\mav_2$ be two free modules of constant finite rank $n$ on $S$. We denote by
$\mav$ the direct sum of $\mav_1$ and $\mav_2$. We denote by $p_1$ and $p_2$ the projections respectively on the first 
and the second factor of this direct sum. We denote by $G$ the group scheme $\op{GL}(\mav_1)\times \op{GL}(\mav_2)$ which is a subgroup scheme of $\op{GL}(\mav)$.

\begin{defi}
We denote by $\Omega:=\op{Iso}(\mav_2,\mav_1)$ the scheme over $S$ parametrizing the isomorphisms from $\mav_2$ to $\mav_1$. 
\end{defi}
There is a natural action of the group scheme $G$ on $\Omega$, via the following formulas 
$$\forall (g_1,g_2)\in G,\quad \forall x\in \Omega,\quad (g_1,g_2)\cdot x=g_1 x g_2^{-1}$$
For this action, $\Omega$ is a homogeneous space under the action of $G$.
\begin{defi}
We denote by $\gr$ the grassmannian
$$\pi : \mathcal{G}r_S(n,\mav)\rw S$$ 
parametrizing the submodules
of $\mav$ which are locally direct summands of rank $n$. 
We denote by $\taut$ the tautological module on $\gr$.  
\end{defi}
The module $\taut$ is a submodule of $\pi^*\mav$ which is locally a direct summand of finite constant rank $n$.
There is a natural action of the group scheme $\op{GL}(\mav)$, and therefore of the group scheme $G$, on the grassmannian $\gr$.
Moreover, $\Omega$ is contained in $\gr$ as a $G$-stable open subscheme via the graph
$$\Omega\rw \gr,\quad x\mapsto \op{Graph}(x).$$
\begin{defi}\label{p}
We denote by $p$ the following morphism of modules on the grassmannian $\gr$ : 
$$p=\pi^*p_1\oplus \pi^*p_2 : \taut^{\oplus 2} \rw \pi^*\mav.$$
\end{defi}

\begin{defi}\label{hd}
For $d\in[\![0,n]\!]$, we denote by $\mah_d$ the locally free module
$$\mah om(\bigwedge^{n+d}(\taut^{\oplus 2}),\bigwedge^{n+d}(\pi^*\mathcal{V})).$$
on the grassmannian $\gr$.
\end{defi}

\begin{defi}
For $d\in[\![0,n]\!]$, the exterior power $\wedge^{n+d}p$ is a global section of $\mah_{d}$. We denote by $\maz_d$ the zero locus of $\wedge^{n+d}p$ on the grassmannian $\gr$.
 
\end{defi}
We define in this way a sequence of $G$-stable closed subschemes on the grassmannian $\gr$
$$\maz_0\subset \maz_1 \subset\cdots\subset \maz_{n}\subset \gr.$$
Observe that the closed subscheme $\maz_0$ is actually empty. 
Moreover, it is easy to prove that the open subscheme $\Omega$ is the complement of $\maz_{n}$ in $\gr$.

\medskip

We will now define a sequence of blow-ups
$$\begin{tikzpicture}
\node (A) at (-1,0) {$\gr_{n}$};  
\node (B) at (2,0) {$\gr_{n-1}$};
\node (C) at (5,0) {$\ldots$};
\node (D) at (7.5,0) {$\gr_1$};
\node (E) at (10,0) {$\gr_0$};
\draw[->,>=latex] (A) to (B);
\draw[->,>=latex] (B) to (C);
\draw[->,>=latex] (C) to (D);
\draw[->,>=latex] (D) to (E);

\draw (0.5,0.2) node[font=\small] {$b_n$};
\draw (8.75,0.2) node[font=\small] {$b_1$};
\end{tikzpicture}$$
and, for each integer $m$ between $0$ and $n$, a family of closed subschemes
$\maz_{m,d}$ of $\gr_m$, where $d$ runs from $m$ to $n$.  

\begin{defi}
Let $m\in[\![0,n]\!]$ and $d\in[\![m,n]\!]$. The definition is by induction:
\begin{itemize}
\item For $m=0$, we set $\gr_0:=\gr$ and $\maz_{0,d}:=\maz_d$.

\item Assuming that the scheme $\gr_{m-1}$ and its subschemes $\maz_{m-1,d}$ are defined, we define 
$$b_m : \gr_m\rw \gr_{m-1}$$ to be the blow-up centered at 
$\maz_{m-1,m}$ and, for each integer $d$ from $m$ to $n$, we define $\maz_{m,d}$ to be the strict transform of $\maz_{m-1,d}$ that is, the schematic closure 
of $$b_m^{-1}(\maz_{m-1,d}\setminus \maz_{m-1,m})$$ in $\gr_m$. 
 
\end{itemize}
Moreover, we denote by 
 $\mathcal{I}_{m,d}$ the ideal sheaf on $\gr_m$ defining $\maz_{m,d}$.
\end{defi} 
 
 The group scheme $G$ acts on the schemes $\gr_m$ and leaves the subschemes $\maz_{m,d}$ globally invariant. 
Modulo Proposition \ref{atlas} below, we prove now:

\begin{thm}\label{comp}
For each integer $m$ from $0$ to $n-1$, the $S$-scheme $\gr_m$ is a smooth projective compactification of $\Omega$.
\end{thm} 
\begin{proof}
By Proposition \ref{atlas} the scheme $\gr_m$ is covered by a collection of open subschemes isomorphic to affine spaces over $S$. In particular, the $S$-scheme 
$\gr_m$ is smooth. It is a classical fact that the grassmannian $\gr$ is projective over $S$. As the blow-up of a projective scheme over $S$ along a closed subscheme is 
projective over $S$, we see that $\gr_m$ is projective over $S$. Finally, observe that the open subscheme $\Omega$ of $\gr$ is disjoint from the closed subscheme 
$\maz_n$ and therefore from each of the closed subscheme $\maz_{d}$. As a consequence, $\Omega$ is an open subscheme of each of the $\gr_m$. 
 
\end{proof}

\subsection{An atlas of affine charts for $\gr_m$}\label{sectionatlas}

Let $V$ be the set $[\![1,n]\!]\times\{1,2\}$. We denote by $V_1$ the subset $[\![1,n]\!]\times\{1\}$ and by $V_2$ the subset $[\![1,n]\!]\times\{2\}$.
We shall refer to elements of $V_1$ as elements of $V$ of type $1$ and elements of $V_2$ as elements of type $2$.
We fix a basis $v_i$, $i\in V$, of the free module $\mathcal{V}$. We suppose that $v_i$, $i\in V_1$ is a basis for $\mav_1$ and 
$v_i$, $i\in V_2$ is a basis for $\mav_2$. Moreover, for each subset $I$ of $V$, we denote by $\mav_{I}$ the free submodule of $\mathcal{V}$ spanned by the 
$v_i$s, where $i$ runs over $I$. For every integer $m$ from $1$ to $n$, we denote by $V^{> m}$ the set $[\![m+1,n]\!]\times\{1,2\}$.
We define the sets $V^{\geqslant m}$, $V^{<m}$ and $V^{\leqslant m}$ similarly. We also have, with obvious notations, the sets $V_1^{>m}$, 
$V_2^{> m}$, $V_1^{\geqslant m}$, 
$V_2^{\geqslant m}\ldots$

\medskip

One word on terminology. If $X$ is an $S$-scheme, by a point $x$ of $X$ we mean an $S$-scheme $S'$ and a point $x$ of the set $X(S')$. However, as it is usually unnecessary, 
we do not mention the $S$-scheme $S'$ and simply write: let $x$ be a point of $X$. 

\begin{defi}
We denote by $R$ the set of permutations $f$ of $V$ such that, for each integer $m$ from $1$ to $n$, the elements $f(m,1)$ and $f(m,2)$ of $V$ have different types.
\end{defi}

\begin{defi}
Let $f\in R$. We denote by $\mau_f$ the affine space  
 $$\op{Spec}(\mathcal{O}_S[x_{i,j}, (i,j)\in f(V_1)\times f(V_2)])$$
 over $S$. It is equipped with a structural morphism $\pi_f$ to $S$.
 and by $\maf_f$ the closed subscheme 
 $$\op{Spec}(\mathcal{O}_S[x_{i,j}, (i,j)\in (f(V_1)_1\times f(V_2)_2)\sqcup (f(V_1)_2\times f(V_2)_1)])$$
\end{defi}
We think of a point $x$ of $\mau_f$ as a matrix indexed by the set $f(V_1)\times f(V_2)$. For a subset $I_1$ of $f(V_1)$ and $I_2$ of $f(V_2)$, we denote 
by $x_{I_1,I_2}$ the submatrix of $x$ indexed by $I_1\times I_2$. For example, the closed subscheme $\maf_f$ is defined by the vanishing of the two matrices 
$x_{f(V_1)_1\times f(V_2)_1}$ and $x_{f(V_1)_2\times f(V_2)_2}$.

\begin{propdef}\label{0open}
Let $f\in R$. There exists a unique morphism  
 $$\iota_f : \mau_f \rw \gr$$ 
such that $\taut_f:=\iota_f^*\taut$ is the submodule of $\pi_f^*\mav$ spanned by  
$$\pi_f^*v_j+\sum_{i\in f(V_1)}x_{i,j}\pi_f^*v_i$$
where $j$ runs over the set $f(V_2)$. The morphism $\iota_f$ is an open immersion. 
We denote by $\gr_f$ the image of the open immersion $\iota_f$. The open subscheme $\gr_f$ cover the grassmannian $\gr$ as $f$ runs over the set $R$.
 \end{propdef}
\begin{proof}
This is classical. The open subscheme $\gr_f$ of the grassmannian parametrizes the complementary submodules of $\mav_{f(V_1)}$ in $\mav$.
\end{proof}

\begin{defi}
Let $f\in R$. We denote by $P_{f,0}$ the subgroup scheme 
$$\op{Stab}_G(\mav_{f(V_1)})$$
of $G$. It is a parabolic subgroup. We also denote by $L_{f,0}$ its Levi subgroup  
$$L_{f,0}:=\op{Stab}_G(\mav_{f(V_1)},\mav_{f(V_2)})=\prod_{i,j\in\{1,2\}}\op{GL}(\mav_{f(V_i)_j})$$
 
\end{defi}
In the next proposition we describe the local structure of the action of the group scheme $G$ on $\gr$. This is analogous to Proposition \ref{key1}.
\begin{prop}\label{0formulas}
Let $f\in R$. The open subscheme $\gr_f$ of $\gr$ is left stable under the action of $P_{f,0}$. For the corresponding 
action of $P_{f,0}$ on $\mau_f$ through the isomorphism $\iota_f$, the closed subscheme $\maf_f$ is left stable under the action of 
$L_{f,0}$ and we have the following formulas
\begin{equation*}
\forall g\in L_{f,0},\quad \forall x\in\maf_f, \quad x'=g\cdot x\text{ where }
\begin{cases} 
x'_{f(V_1)_1,f(V_2)_2}=g_{f(V_1)_1}x_{f(V_1)_1,f(V_2)_2}g_{f(V_2)_2}^{-1}
\\
x'_{f(V_1)_2,f(V_2)_1}=g_{f(V_1)_2}x_{f(V_1)_2,f(V_2)_1}g_{f(V_2)_1}^{-1}
\end{cases}
\end{equation*}
Finally, the natural morphism 
$$m_{f,0} : R_u(P_{f,0})\times \maf_f\rw \mau_f,\quad (r,x)\mapsto r\cdot x$$
is an isomorphism.
\end{prop}
 \begin{proof}
The open subscheme $\gr_f$ of the grassmannian parametrizes the complementary submodules of $\mav_{f(V_1)}$ in $\mav$. It follows that it is stable under the action 
of the stabilizer $P$ of $\mav_{f(V_1)}$ in $\op{GL}(\mav)$ and therefore under the action of its subgroup $P_{f,0}$.
 
\medskip 
Let $x$ be a point of $\mau_f$ and $g$ a point of $P$.
By definition, the point $\iota_{f}(x)$ is the graph of $x$. Therefore,
the point $g\cdot\iota_{f,0}(x)$ is the module consisting  
of elements of type 
$$g(v+xv)=g_{f(V_2)}v+(g_{f(V_1),f(V_2)}+g_{f(V_1)}m)v,\quad v\in\mav_{f(V_2)}.$$
It is thus equal to the point 
$$\iota_{f,0}((g_{f(V_1),f(V_2)}+g_{f(V_1)}x)g_{f(V_2)}^{-1}).$$
In other words, the action of $P$ on $\mau_f$ is given by 
$$P\times \mau_f\rw \mau_f,\quad (g,x)\mapsto (g_{f(V_1),f(V_2)}+g_{f(V_1)}x)g_{f(V_2)}^{-1}.$$
By specializing this action to the subgroup $P_{f,0}$ of $P$, we immediately see that $\maf_f$ is left stable under the action of $L_{f,0}$ we obtain 
the formulas in the statement of the proposition.
\medskip
 
Moreover, still using the description of the action of $P$ on $\mau_f$ found above, we see that if $g$ is a point of $R_u(P_{f,0})$ and $x$ a point of 
$\mau_f$, then the point $x'=g\cdot x$ of 
$\mau_f$ is given by :

\begin{equation*}
\begin{cases} 
x'_{f(V_1)_1,f(V_2)_1}=g_{f(V_1)_1,f(V_2)_1}
\\
x'_{f(V_1)_1,f(V_2)_2}=x_{f(V_1)_1,f(V_2)_2}
\\
x'_{f(V_1)_2,f(V_2)_1}=x_{f(V_1)_2,f(V_2)_1}
\\
x'_{f(V_1)_2,f(V_2)_2}=g_{f(V_1)_2,f(V_2)_2}.
\end{cases}
\end{equation*}
This proves that the natural $P_{f,0}$-equivariant morphism : 
$$m_{f,0} : R_u(P_{f,0})\times \maf_{f,0}\rw \mau_f$$
is indeed an isomorphism.
 
\end{proof}

 \begin{defi}
Let $f\in R$ and $d\in[\![0,n]\!]$. We denote by $\mai_{f,0,d}$ the ideal sheaf on $\maf_{f}$ spanned by the minors of size $d$ of the matrix   
 $$\left(
\begin{array}{cccc}
0 & x_{f(V_1)_1,f(V_2)_2} \\
x_{f(V_1)_2,f(V_2)_1} & 0 \\
\end{array}
\right). $$
We denote by $\maz_{f,0,d}$ the closed subscheme of $\maf_{f}$ defined by the ideal sheaf $\mai_{f,0,d}$.
  
 \end{defi}

\begin{prop}\label{0ideal}
Let $f\in R$ and $d\in[\![0,n-1]\!]$. Through the isomorphism $$m_{f,0} : R_u(P_{f,0})\times \maf_{f,0}\rw \mau_f$$ of Proposition \ref{0formulas} the 
closed subscheme $\iota_f^{-1}(\maz_{0,d})$ is equal to $R_u(P_{f,0})\times \maz_{f,0,d}$.
\end{prop}
\begin{proof}
Due to the formula in the proof of Proposition \ref{0formulas}, it suffices to show that the defining ideal of
$\iota_f^{-1}(\maz_{0,d})$ on $\mau_f$ is spanned by the minors of size $d$ of the matrix   
 $$\left(
\begin{array}{cccc}
0 & x_{f(V_1)_1,f(V_2)_2} \\
x_{f(V_1)_2,f(V_2)_1} & 0 \\
\end{array}
\right). $$
To prove this, we express the matrix of the homomorphism
$$\iota_f^*p : \taut_f^{\oplus 2}\rw \pi_f^*\mathcal{V}$$
in appropriate basis. We choose the basis of $\taut_f$ described in Proposition-Definition \ref{0open}.
This basis is indexed by the set $f(V_2)$, which is the disjoint union of $f(V_2)_1$ and $f(V_2)_2$. We also choose the basis 
$$\pi_f^*(v_i),\quad i\in f(V_1)_1,\quad \pi_f^*(v_j)+\sum_{i\in f(V_1)_1}x_{i,j}\pi_f^*(v_i),\quad j\in f(V_2)_1$$
for $\pi_f^*\mav_1$ and 
$$\pi_f^*(v_i),\quad i\in f(V_1)_2,\quad \pi_f^*(v_j)+\sum_{i\in f(V_1)_2}x_{i,j}\pi_f^*(v_i),\quad j\in f(V_2)_2$$
for $\pi_f^*\mav_2$. The matrix of $\iota_f^*p$ in these basis can be expressed in blocks as follows: 

$$\left(
\begin{array}{cccc}
0 & x_{f(V_1)_1,f(V_2)_2} & 0 & 0 \\
Id & 0 & 0 & 0 \\
0 & 0 & x_{f(V_1)_2,f(V_2)_1} & 0 \\
0 & 0 & 0 & Id \\

\end{array}
\right). $$
By definition, the defining ideal of $\iota_f^{-1}(\maz_{0,d})$ is generated by the minors of size $n+d$ of this matrix.  
By reordering the vector in the basis, we get the block diagonal square matrix with blocks $I_n$ and
 $$\left(
\begin{array}{cccc}
0 & x_{f(V_1)_1,f(V_2)_2} \\
x_{f(V_1)_2,f(V_2)_1} & 0 \\
\end{array}
\right). $$
We see therefore that the defining ideal of $\iota_f^{-1}(\maz_{0,d})$ is generated by the minors of size $d$ of the last matrix, as we wanted.

\end{proof}

\begin{defi}\label{bigdefi}
 Let $f\in R$, $m\in[\![0,n]\!]$ and $d\in[\![m,n]\!]$. 
\begin{itemize}
 \item We define a parabolic subgroup scheme $P_{f,m}$ of $G$ by induction
$m$. For $m$ equals $0$, we have 
already defined 
$P_{f,0}$.
Then, assuming that $P_{f,m-1}$ has been defined, we set
\begin{equation*}
P_{f,m}=
\begin{cases} 
\op{Stab}_{P_{f,m-1}}(\mav_{f(V_1^{>m})\cap V_1},\mav_{\{f(m,2)\}}) & \text{if }f(m,1)\in V_1\text{ and }f(m,2)\in V_2
\\
\op{Stab}_{P_{f,m-1}}(\mav_{f(V_1^{>m})\cap V_2},\mav_{\{f(m,2)\}}) & \text{if }f(m,1)\in V_2\text{ and }f(m,2)\in V_1.
\end{cases}
\end{equation*}
 
\item We denote by $L_{f,m}$ the following Levi subgroup of $P_{f,m}$:
$$\prod_{i=1}^k(\op{GL}(\mav_{f(i,1)})\times \op{GL}(\mav_{f(i,2)}))\times\prod_{i,j\in\{1,2\}}\op{GL}(\mav_{f(V_i^{>m})\cap V_j})$$
 
\item We denote by $\maf_{f,m}$ the affine space over $S$ on the indeterminates 
$x_{i,j}$ where $(i,j)$ runs over the union of the sets 
$$\{(f(1,1),f(1,2)),\ldots,(f(m,1),f(m,2))\}$$
and 
$$((f(V_1^{>m})_1)\times (f(V_2^{>m})_2))\cup ((f(V_1^{>m})_2)\times (f(V_2^{>m})_1)).$$ 
 
\item  We let the group scheme $L_{f,m}$ act on $\maf_{f,m}$ by the following formulas
\begin{equation*}
\begin{cases} 
x'_{f(1,1),f(1,2)}=g_{f(1,1)}g_{f(1,2)}^{-1}x_{f(1,1),f(1,2)}
\\
x'_{f(i,1),f(i,2)}=g_{f(i,1)}g_{f(i-1,2)}g_{f(i-1,1)}^{-1}g_{f(i,2)}^{-1}x_{f(i,1),f(i,2)} \text{ for }i\in[\![2,m]\!]
\\
x'_{f(V_1^{>m})_1,f(V_2^{>m})_2}=g_{f(m,1)}^{-1}g_{f(m,2)}g_{f(V_1^{>m})_1}x_{f(V_1^{>m})_1,f(V_2^{>m})_2}g_{f(V_2^{>m})_2}^{-1}
\\
x'_{f(V_1^{>m})_2,f(V_2^{>m})_1}=g_{f(m,1)}^{-1}g_{f(m,2)}g_{f(V_1^{>m})_2}x_{f(V_1^{>m})_2,f(V_2^{>m})_1}g_{f(V_2^{>m})_1}^{-1}

\end{cases}
\end{equation*}
where $g$ is a point of $L_{f,m}$, $x$ a point of $\maf_{f,m}$ and $x':=g\cdot x$.

\item We denote by $\mau_{f,m}$ the product $$R_u(P_{f,m})\times \maf_{f,m}$$ acted upon by the group scheme $P_{f,m}=R_u(P_{f,m})\rtimes L_{f,m}$
via the formula
$$\forall (r,l)\in P_{f,m},\quad \forall (r',x)\in\mau_{f,m},\quad (r,l)\cdot (r',x)=(rlr'l^{-1},l\cdot x).$$ 

\item  We denote by $\mai_{f,m,d}$ the ideal sheaf on $\maf_{f,m}$ spanned by the minors of size $d-m$ of the matrix   
 $$\left(
\begin{array}{cccc}
0 & x_{f(V_1^{> m})_1,f(V_2^{> m})_2} \\
x_{f(V_1^{> m})_2,f(V_2^{> m})_1} & 0 \\
\end{array}
\right). $$
 
\item  We denote by $\maz'_{f,m,d}$ the closed subscheme of $\maf_{f,m}$ defined by the ideal sheaf $\mai_{f,m,d}$ and by $\maz_{f,m,d}$ the closed subscheme 
$R_u(P_{f,m,d})\times \maz'_{f,m,d}$ of $\mau_{f,m}$.

\item We denote by $\mathcal{B}_{f,m}$ the blow-up of $\mau_{f,m}$ along the closed subscheme $\maz_{f,m,m+1}$.
 
\end{itemize}
 \end{defi}
Let $f\in R$, $m\in[\![1,n]\!]$ and $d\in[\![m,n]\!]$.
The blow-up $\mathcal{B}_{f,m-1}$ is the closed subscheme of  
$$\mau_{f,m-1}\times \op{Proj}(\mao_S[X_{i,j},(i,j)\in (f(V_1^{\geqslant m})_1\times f(V_2^{\geqslant m})_2) \sqcup (f(V_1^{\geqslant m})_2\times f(V_2^{\geqslant m})_1)])$$
defined by the equations 
$$\forall (i,j), (i',j')\in (f(V_1^{\geqslant m})_1\times f(V_2^{\geqslant m})_2) \sqcup (f(V_1^{\geqslant m})_2\times f(V_2^{\geqslant m})_1)$$
$$x_{i,j}X_{i',j'}-X_{i,j}x_{i',j'}=0.$$ 
 
\begin{prop}\label{key} 
 With these notations, the open subscheme 
$\{X_{f(m,1),f(m,2)}\neq 0\}$ 
of $\mathcal{B}_{f,m-1}$ is left stable under the action of $P_{f,m}$ and is isomorphic, as a $P_{f,m}$-scheme, to 
$\mau_{f,m}$. Moreover, via this isomorphism, the strict transform of $\maz_{f,m-1,d}$ in $\mau_{f,m}$ is $\maz_{f,m,d}$.
\end{prop}

\begin{proof}
We prove analogous statement for the 
the blow-up $\mab'_{f,m-1}$ of $\maf_{f,m-1}$ along the closed subscheme $\maz'_{f,m-1,m}$ from which the proposition is easily derived. 

\medskip
The scheme $\mab'_{f,m-1}$ is the closed subscheme of  
$$\maf_{f,m-1}\times \op{Proj}(\mao_S[X_{i,j},(i,j)\in (f(V_1^{\geqslant m})_1\times f(V_2^{\geqslant m})_2) \sqcup (f(V_1^{\geqslant m})_2\times f(V_2^{\geqslant m})_1)])$$
defined by the equations 
$$\forall (i,j), (i',j')\in (f(V_1^{\geqslant m})_1\times f(V_2^{\geqslant m})_2) \sqcup (f(V_1^{\geqslant m})_2\times f(V_2^{\geqslant m})_1)$$
$$x_{i,j}X_{i',j'}-X_{i,j}x_{i',j'}=0.$$ 
Observe that the open subscheme $\mau'_{f,m}$ of $\mab'_{f,m-1}$ defined by $X_{f(m,1),f(m,2)}\neq 0$ is isomorphic, as a scheme over $\maf_{f,m-1}$ to 
\begin{equation*}
b : \maf_{f,m-1}\rw \maf_{f,m-1},\quad x_{i,j}\mapsto
\begin{cases} 
x_{f(m,1),f(m,2)}x_{i,j} &\text{if }(i,j)\in f(V_1^{>m})\times f(V_2^{>m}) \text{ and}
\\
 &\text{$i$ and $j$ have different types}
\\
x_{i,j} &\text{otherwise.}
\end{cases}
\end{equation*}
In the following we shall make use of this remark and use the coordinates $x_{i,j}$ to describe the points of $\mau'_{f,m}$.

\medskip

Observe that the center of the blow-up $\mab'_{f,m-1}\rw \maf_{f,m-1}$ is stable under the action of $L_{f,m-1}$ and therefore the group scheme $L_{f,m-1}$ acts on $\mab'_{f,m-1}$.  
This action can be described as follows.
$$\forall g\in L_{f,m-1},\quad \forall (x,X)\in \mab'_{f,m}, \quad (x',X')=g\cdot (x,X)\text{ where } x'=g\cdot x\text{ and} $$
\begin{equation*}
\begin{cases} 
X'_{f(V_1^{\geqslant m})_1,f(V_2^{\geqslant m})_2}=g_{f(m-1,1)}^{-1}g_{f(m-1,2)}g_{f(V_1^{\geqslant m})_1}X_{f(V_1^{\geqslant m})_1, f(V_2^{\geqslant m})_2}g_{f(V_2^{\geqslant m})_2}^{-1}
\\
X'_{f(V_1^{\geqslant m})_2,f(V_2^{\geqslant m})_1}=g_{f(m-1,1)}^{-1}g_{f(m-1,2)}g_{f(V_1^{\geqslant m})_2}X_{f(V_1^{\geqslant m})_2, f(V_2^{\geqslant m})_1}g_{f(V_2^{\geqslant m})_1}^{-1}
\end{cases}
\end{equation*}
where we denote by $X$ the matrix formed by the $X_{i,j}$. 
Observe now that the open subscheme $\mau_{f,m}'$ is the locus of points $(x,X)$ of $\mab'_{f,m-1}$ such that $v_{f(m,1)}^*(Xv_{f(m,2)})$ does not vanish.
Therefore, it is left stable under the action of the parabolic subgroup of $L_{f,m-1}$ 
\begin{equation*}
P=
\begin{cases} 
\op{Stab}_{L_{f,m-1}}(\mav_{f(V^{>m}_1)},\mav_{f(m,2)}) & \text{ if $f(m,1)$ belongs to $V_1$ and $f(m,2)$ to $V_2$}
\\
\op{Stab}_{L_{f,m-1}}(\mav_{f(V^{>m}_2)},\mav_{f(m,2)}) & \text{ if $f(m,1)$ belongs to $V_2$ and $f(m,2)$ to $V_1$}
\end{cases}
\end{equation*}

\medskip
The group scheme $L_{f,m}$ is a Levi subgroup scheme of $P$.
Let $g$ be a point of $L_{f,m}$ and $x$ a point of $\maf_{f,m}$. The point $x$ corresponds to the couple $(b(x),X)$ in $\mab'_{f,m-1}$, where 
\begin{equation*}
X_{i,j} =
\begin{cases} 
x_{i,j} &\text{ if }(i,j)\neq (f(m,1),f(m,2))
\\
1 &\text{ if }(i,j)=(f(m,1),f(m,2)).
\end{cases}
\end{equation*}
Let $(x',X')=g\cdot (x,X)$. By a quick computation we get
\begin{equation*}
\begin{cases} 
x'_{f(m,1),f(m,2)}= g_{f(m-1,1)}^{-1}g_{f(m-1,2)}g_{f(m,1)} g_{f(m,2)}^{-1} x_{f(m,1),f(m,2)}
\\
X'_{f(m,1),f(m,2)} = g_{f(m-1,1)}^{-1}g_{f(m-1,2)}g_{f(m,1)} g_{f(m,2)}^{-1} 
\\
X'_{f(V^{>m}_1)_1,f(V^{>m}_2)_2}=g_{f(m-1,1)}^{-1}g_{f(m-1,2)}g_{f(V^{>m}_1)_1}X_{f(V^{>m}_1)_1,f(V^{>m}_2)_2}g_{f(V^{>m}_2)_2}^{-1}
\\
X'_{f(V^{>m}_1)_2,f(V^{>m}_2)_1}=g_{f(m-1,1)}^{-1}g_{f(m-1,2)}g_{f(V^{>m}_1)_2}X_{f(V^{>m}_1)_2,f(V^{>m}_2)_1}g_{f(V^{>m}_2)_1}^{-1}
\\
X'_{i,j} =0 \text{ otherwise }
\end{cases}
\end{equation*}
Therefore we see that the closed subscheme $\maf_{f,m}$ of $\mau'_{f,m}$ is left stable under the action of the Levi subgroup $L_{f,m}$ and, moreover, the action of
$L_{f,m}$ on $\maf_{f,m}$ is given by the formulas in Definition \ref{bigdefi}. 
In a similar way, we prove that 
 the natural morphism 
$$R_u(P)\times \maf_{f,m}\rw \mau'_{f,m},\quad (g,x)\mapsto g\cdot x=x'$$ is an isomorphism, given by the following formulas
\begin{equation*}
\begin{cases} 
x'_{f(l,1),f(l,2)}= x_{f(l,1),f(l,2)} \text{ for all $l\in[\![1,m]\!]$}
\\
x'_{f(V^{>m}_1)_1,f(m,2)} = g_{f(V^{>m}_1)_1,f(m,1)}
\\
x'_{f(m,1),f(V^{>m}_2)_2} = -g_{f(m,2),f(V^{>m}_2)_2}
\\
x'_{f(V^{>m}_1)_1,f(V^{>m}_2)_2}=x_{f(V^{>m}_1)_1,f(V^{>m}_2)_2}-g_{f(V^{>m}_1)_1,f(m,1)}g_{f(m,2),f(V^{>m}_2)_2}
\\
x'_{f(V^{>m}_1)_2,f(V^{>m}_2)_1}=x_{f(V^{>m}_1)_2,f(V^{>m}_2)_1}

\end{cases}
\end{equation*}

\medskip
Now we compute the strict transform of $\maz'_{f,m,d}$ in $\mau'_{f,m}$.
By definition, the strict transform of $\maz'_{f,m,d}$ is the schematic closure of 
$$\maz:=b^{-1}(\maz'_{f,m-1,d}\cap \{x_{f(m,1),f(m,2)}\neq 0\})$$ 
in $\mau'_{f,m}$. Let $x$ be a point of $\mau'_{f,m}$. We denote $(r,y)$ its components through the isomorphism
$$R_u(P)\times \maf_{f,m}\rw \mau'_{f,m}.$$
By definition, the point $x$ belongs to $\maz$ if and only $x_{f(m,1),f(m,2)}$ is invertible and all the square $d-m+1$ minors extracted 
from the following matrix 
$$\left(
\begin{array}{cccc}
0 & b(x_{f(V_1^{\geqslant m})_1,f(V_2^{\geqslant m})_2}) \\
b(x_{f(V_1^{\geqslant m})_2,f(V_2^{\geqslant m})_1}) & 0 \\
\end{array}
\right) $$
are zero.
 By definition of the morphism $b$, each coefficient in this matrix is a multiple of $x_{f(m,1),f(m,2)}$ and the coefficient in place $(f(m,1),f(m,2))$ is exactly 
$x_{f(m,1),f(m,2)}$. As this indeterminate is invertible on the open subscheme $\{x_{f(m,1),f(m,2)}\neq 0\}$ we see that the point $x$ belongs to $\maz$ 
if and only $x_{f(m,1),f(m,2)}$ is invertible and all the minors of size $d-m+1$ extracted from the matrix 
$$\left(
\begin{array}{cccc}
0 & x'_{f(V_1^{\geqslant m})_1,f(V_2^{\geqslant m})_2}   \\
x'_{f(V_1^{\geqslant m})_2,f(V_2^{\geqslant m})_1} & 0 \\
\end{array}
\right) $$
 are zero, where $x'_{f(m,1),f(m,2)}=1$ and $x'_{i,j}=x_{i,j}$ otherwise. By operating standard row and column operations on this matrix we see now that $x$ belongs to $\maz$ 
 if and only if $x_{f(m,1),f(m,2)}$ is invertible and all the minors of size $d-m$ extracted from the matrix 
 $$\left(
\begin{array}{cccc}
0 & y_{f(V_1^{> m})_1,f(V_2^{> m})_2} \\
y_{f(V_1^{> m})_2,f(V_2^{> m})_1} & 0 \\
\end{array}
\right) $$
are zero.
This proves that $\maz$ is the intersection of $R_u(P)\times \maz'_{f,m,d}$ with the open subscheme $\{x_{f(m,1),f(m,2)}\neq 0\}$.
 
\medskip 
 
In order to complete the proof, we now show that the schematic closure of $\maz$ in 
$\mau'_{f,m}$ is equal to $R_u(P)\times \maz'_{f,m,d}$. 
We can check this Zariski locally on the base $S$ and therefore assume that $S$ is affine. First of all, it is obvious that the closure of 
$\maz$ is 
contained in $R_u(P)\times \maz'_{f,m,d}$. Conversely, let $\varphi$ be a global function on $\mau'_{f,m}$ which vanishes on $\maz$. 
This means that there exists an integer $q$ such
that $x_{f(m,1),f(m,2)}^q\varphi$ belongs to the ideal spanned by the minors of size $d-m$  
of the following matrix: 
$$\left(
\begin{array}{cccc}
0 & x_{f(V_1^{>m})_1,f(V_2^{>m})_2} \\
x_{f(V_1^{>m})_2,f(V_2^{>m})_1} & 0 \\
\end{array}
\right).$$
As the indeterminate $x_{f(m,1),f(m,2)}^q\varphi$ does not appear in this matrix, we can conclude that $\varphi$ itself belongs to this ideal, completing the proof of the 
proposition. 
\end{proof}

 \begin{prop}\label{atlas}
Let $f\in R$. There exists a unique collection of open immersions $\iota_{f,m}$, for $m$ from $1$ to $n$
such that each of the squares below are commutative :
$$\begin{tikzpicture}
\node (A) at (0,0) {$\gr_{0}$};
\node (B) at (0,2.7) {$\mau_{f,0}$};
\node (C) at (2.7,0) {$\gr_1$};
\node (D) at (2.7,2.7) {$\mau_{f,1}$};
\node (G) at (6.1,0) {$\gr_{n-1}$};
\node (H) at (6.1,2.7) {$\mau_{f,n-1}$};
\node (I) at (8.9,0) {$\gr_n$};
\node (J) at (8.9,2.7) {$\mau_{f,n}$};

\node at (4.5,0) {$.$};
\node at (4.4,0) {$.$};
\node at (4.3,0) {$.$};
\node at (4.5,2.7) {$.$};
\node at (4.4,2.7) {$.$};
\node at (4.3,2.7) {$.$};

\draw[->,>=latex] (B) to (A);
\draw[->,>=latex] (D) to (C);
\draw[->,>=latex] (H) to (G);
\draw[->,>=latex] (J) to (I);

\draw[->,>=latex] (C) to (A);
\draw[->,>=latex] (D) to (B);
\draw[->,>=latex] (I) to (G);
\draw[->,>=latex] (J) to (H);

\draw (G) to (4.6,0);
\draw[->,>=latex] (4.2,0) to (C);
\draw (H) to (4.6,2.7);
\draw[->,>=latex] (4.2,2.7) to (D);

\draw (-0.4,1.35) node[font=\small] {$\iota_{f,0}$};
\draw (3.1,1.35) node[font=\small] {$\iota_{f,1}$};
\draw (5.5,1.35) node[font=\small] {$\iota_{f,n-1}$};
\draw (9.35,1.35) node[font=\small] {$\iota_{f,n}$};

\draw (1.4,2.9) node[font=\small] {$b_{f,1}$};
\draw (1.4,-0.3) node[font=\small] {$b_{1}$};
 
\draw (7.5,2.9) node[font=\small] {$b_{f,n}$};
\draw (7.5,-0.3) node[font=\small] {$b_{n}$};

\end{tikzpicture}$$
The open immersion $\iota_{f,m}$ is equivariant for the action of $P_{f,m}$. We denote by $\gr_{f,m}$ the image of the open immersion $\iota_{f,m}$. 
The open subschemes $\gr_{f,m}$ cover $\gr_m$ as $f$ run over $R$.
\end{prop}

\begin{proof}
The existence and $P_{f,m}$-equivariance of the $\iota_{f,m}$ follows directly from Proposition \ref{key}.
The uniqueness comes from the fact that for each index $m$, the intersection $\Omega\cap \gr_{f,m}$ is dense in $\gr_{f,m}$.
The last assertion is proved as follows. By Proposition \ref{0open}, the open subschemes $\gr_{f,0}$ cover the scheme $\gr_0$. 
Moreover, it follows from Proposition \ref{key} that the blow-up $\mab_{f,m-1}$ is covered by the open subschemes $\mau_{f',m}$ where 
$f'$ runs over the elements of $R$ having the same restriction as $f$ to $[\![1,k]\!]\times \{1,2\}$ and satisfying 
$f(V_1)=f'(V_1)$ and $f(V_2)=f'(V_2)$.  
\end{proof}

\subsection{An alternative construction}\label{alternative}
In this section we provide an alternative construction for the schemes $\gr_m$.
Recall from Section \ref{defofcomp} that the section $\wedge^{n+d}p$ of the locally free module $\mah_d$ does not vanish on $\Omega$.
Therefore it defines an invertible submodule of $\mah_d$ on $\Omega$ that is locally a direct summand. In other words, it defines a morphism from
$\Omega$ to the projective bundle $\mathbb{P}(\mah_d)$ over $\Omega$. 
\begin{defi}
Let $d\in[\![0,n]\!]$. We denote by $\varphi_d$ the morphism  
$$\varphi_d : \Omega\rw \mathbb{P}(\mah_d).$$ 
defined by the global section $\wedge^{n+d} p$ of $\mah_d$.
\end{defi}

\begin{prop}\label{extend}
Let $d\in[\![0,n]\!]$ and $m\in[\![d,n]\!]$. The morphism $\varphi_d$ extends to $\gr_m$ in a unique way.
\end{prop}
\begin{proof}
We first observe that $\Omega$ is dense in each of the schemes $\gr_m$. If the morphism $\varphi_d$ extends to $\gr_m$ it is therefore in a unique
way. Moreover, it suffices to show that $\varphi_d$ extends to $\gr_d$, because then the composite  
$$\begin{tikzpicture}
\node (A) at (-1,0) {$\gr_{m}$};  
\node (B) at (2,0) {$\gr_{m-1}$};
\node (C) at (5,0) {$\ldots$};
\node (D) at (7.5,0) {$\gr_d$};
\node (E) at (10.5,0) {$\mathbb{P}(\mah_d)$};
\draw[->,>=latex] (A) to (B);
\draw[->,>=latex] (B) to (C);
\draw[->,>=latex] (C) to (D);
\draw[->,>=latex] (D) to (E);

\draw (0.5,0.3) node[font=\small] {$b_{m}$};
\draw (8.75,0.2) node[font=\small] {$\varphi_d$};
\end{tikzpicture}$$
is an extension of $\varphi_d$ to $\gr_m$.

\medskip

By the same density argument as above it suffices to show that the morphism $\varphi_d$ extends from $\Omega\cap \gr_{f,d}$ to $\gr_{f,d}$ for each element 
$f$ of the set $R$. We identify the scheme $\gr_{f,d}$ with $\mau_{f,d}$ via the isomorphism $\iota_{f,d}$.
Observe that the scheme $\mathbb{P}(\mah_d)$ is equipped with an action of the group scheme $G$ and the morphism $\varphi_d$ is equivariant 
with respect to this action. By using this remark, we see that it suffices to extend the morphism $\varphi_d$ from 
$\Omega\cap\maf_{f,d}$ to $\maf_{f,d}$. 

\medskip
We denote by $c$ the composite $b_{f,1}\circ\cdots\circ b_{f,d}$. We use the trivialization 
of $\mah_{d}$ on $\gr_{f,0}$ as in the proof of Proposition \ref{0ideal}. For this trivialization, the coordinates are indexed by the product 
$(V_2\sqcup V_2)\times V$. The morphism $\varphi_d$ is given over $\Omega\cap \maf_{f,d}$ in this coordinates by the $n+d$ minors of the matrix 
$$\left(
\begin{array}{cccc}
0 & c(x)_{f(V_1)_1,f(V_2)_2} & 0 & 0 \\
Id & 0 & 0 & 0 \\
0 & 0 & c(x)_{f(V_1)_2,f(V_2)_1} & 0 \\
0 & 0 & 0 & Id \\

\end{array}
\right). $$
It follows from the definition of $c$ that all these minors are multiples of  
$$x_{f(1,1),f(1,2)}^d x_{f(1,2),f(2,2)}^{d-1}\ldots x_{f(d,1),f(d,2)}$$
and that one of them, namely the one indexed by the product of 
$$f(V_1)_2\sqcup f(V_2^{\leqslant d})_2\sqcup (f(V_1^{\leqslant d})_2)\sqcup f(V_2)_2$$
and
$$f(V_1^{\leqslant d})_1\sqcup f(V_2)_1\sqcup f(V_1^{\leqslant d})_2 \sqcup f(V_2)_2$$
is exactly equal to this product or its opposite. By dividing each coordinate by this product we therefore extend the morphism 
$\varphi_d$ to $\maf_{f,d}$.
\end{proof}

\begin{prop}
Let $m\in[\![0,n]\!]$. The morphism 
$$\psi_m:=\varphi_{0}\times \varphi_{1}\times ...\times \varphi_{m} : \gr_m\rw \mathbb{P}(\mathcal{H}_{0})\times_{\gr}\mathbb{P}(\mathcal{H}_{1})\times_{\gr}...\times_{\gr}\mathbb{P}(\mathcal{H}_{m})$$
is a closed immersion.  
\end{prop}

\begin{proof}
 
We prove this by induction on $m$. For $m=0$ the morphism $\varphi_0$ is defined by the nowhere vanishing section $\wedge^n p$ of $\mah_0$ 
and is therefore a closed immersion. We suppose now that $\psi_{m-1}$ is a closed immersion and prove that $\psi_m$ is also a closed immersion.
As this morphism is proper, it suffices to check that it is a monomorphism  
in order to prove that it is a closed immersion. 
Let $p_1$ and $p_2$ be two points of $\gr_m$ which are mapped to the same point by $\psi_m$. We want to show that they are equal.
\medskip

First, we suppose that $p_1$ and $p_2$ are points of $\gr_{f,m}$, where $f$ is an element of $R$. We identify $\gr_{f,m}$ and $\mau_{f,m}$  
via the isomorphism $\iota_{f,m}$. We use the notations introduced in the proof of Proposition \ref{key}.
The scheme $\mau_{f,m}$ is isomorphic to the product $R_u(P_{f,m-1})\times \mau'_{f,m}$. Using the induction hypothesis, we can assume that 
$p_1$ and $p_2$ are actually points of $\mau'_{f,m}$. 
Viewed as points of $\mau'_{f,m}$ (which is isomorphic to $\maf_{f,m-1}$ via the morphism $b$) we denote the coordinates of $p_1$ by 
$(x_{i,j,1})$ and and the coordinates of $p_2$ by $(x_{i,j,2})$. Still by the induction hypothesis, we have 
$x_{i,j,1}=x_{i,j,2}$
for $(i,j)$ in the following set 
$$\{(f(1,1),f(1,2)),\ldots, (f(m,1),f(m,2))\}$$ 
as it follows from the definition of the morphism $b$. Observe now that the morphism $\varphi_m$ is defined over $\mau'_{f,m}$ by exactly the same process as explained 
in the proof of Proposition \ref{extend}. By this we mean that the coordinates of $\varphi_d$ are obtained by computing the $n+d$ minors of the matrix  
$$\left(
\begin{array}{cccc}
0 & c(x)_{f(V_1)_1,f(V_2)_2} & 0 & 0 \\
Id & 0 & 0 & 0 \\
0 & 0 & c(x)_{f(V_1)_2,f(V_2)_1} & 0 \\
0 & 0 & 0 & Id \\

\end{array}
\right)$$
and dividing by the product 
$$x_{f(1,1),f(1,2)}^m x_{f(1,2),f(2,2)}^{m-1}\ldots x_{f(m,1),f(m,2)}.$$
Indeed this process makes sense over $\mau'_{f,m}$ and extend $\varphi_d$ over $\Omega\cap \mau'_{f,m}$. Moreover, such an extension is unique.
Let $(i,j)$ be an element of 
$$(f(V_1^{\geqslant m})_1\times f(V_2^{\geqslant m})_2)\sqcup (f(V_1^{\geqslant m})_2\times f(V_2^{\geqslant m})_1)$$
different from $(f(m,1),f(m,2))$. We suppose that $i$ is of type $1$ and $j$ of type $2$, the other case being entirely similar.
The coordinate of $\varphi_m$ corresponding to the minor indexed by the product of 

$$f(V_1)_2\sqcup f(V_2^{< d})_2\sqcup\{j\}\sqcup (f(V_1^{\leqslant d})_2)\sqcup f(V_2)_2$$
and
$$f(V_1^{\leqslant d})_1\sqcup f(V_2)_1\sqcup f(V_1^{< d})_2\sqcup \{i\} \sqcup f(V_2)_2.$$
is $x_{i,j}$ or its opposite. We can therefore conclude that $x_{i,j,1}=x_{i,j,2}$. Finally we have proved that $p_1=p_2$.

\medskip
We go back to the general case. Let $f$ be an element of $R$. We prove now that the open subschemes $p_1^{-1}(\gr_{f,m})$ and $p_2^{-1}(\gr_{f,m})$ are equal.
This is sufficient to complete the proof of the proposition.
By the induction hypothesis, we already know that the open subschemes $p_1^{-1}(\gr_{f,m-1})$ and $p_2^{-1}(\gr_{f,m-1})$ are equal.
Observe now that the computations above actually prove the following: through $\varphi_m$, the non zero locus on $\mab_{f,m}$
of the coordinate of $\mathbb{P}(\mah_m)$ corresponding to the minor indexed by the product of 
$$f(V_1)_2\sqcup f(V_2^{\leqslant m})_2\sqcup (f(V_1^{\leqslant m})_2)\sqcup f(V_2)_2$$
and
$$f(V_1^{\leqslant m})_1\sqcup f(V_2)_1\sqcup f(V_1^{\leqslant m})_2 \sqcup f(V_2)_2$$
is $\mau_{f,m}$. This implies the result. 

\end{proof}

\subsection{The colored fan of $\gr_n$}\label{colored}

In this section we assume that the base scheme $S$ is the spectrum of an algebraically closed field $k$ of arbitrary characteristic.
The scheme $\gr_n$ is an equivariant compactification of the homogeneous space $\op{Iso}(\mav_2,\mav_1)$ under the action of the group 
$\op{GL}(\mav_1)\times \op{GL}(\mav_2)$. Through the fixed trivializations of the free modules $\mav_1$ and $\mav_2$ we see that $\gr_n$ is an
equivariant compactification of the general linear group $\op{GL}(n)$ under the action of $\op{GL}(n)\times\op{GL}(n)$. The aim of this section is to compute 
the colored fan of this compactification, as explained in Section \ref{example}.

\medskip

But first, we say a word about the blow-up procedure explained in Section \ref{defofcomp}  in this setting.
By definition, the set $\maz_d(k)$ is the set of $n$-dimensional subspaces of $\mav(k)$ such that 
the sum of the ranks of $p_1(k)$ and $p_2(k)$ is less than $n+d$ at every point. Another way to state this is that $\maz_d(k)$ is the set 
$$\{F\in \gr(k),\quad \op{dim}(F\cap \mav_1(k))+\op{dim}(F\cap \mav_2(k))\geqslant n-d\}.$$
For example, the set $\maz_0(k)$ is the set of $n$-dimensional subspaces of $\mav(k)$
which are direct sum of a subspace of $\mav_1(k)$ and a subspace of $\mav_2(k)$.
Using this description, it is not difficult to prove that $\maz_0(k)$ is the union of the closed orbits of $G(k)$ in $\gr(k)$. We leave it as an exercise to the reader 
to prove that, for $d$ from $1$ to $n-1$, an orbit $\omega$ of $G(k)$ in $\gr(k)$ is contained in $\maz_d(k)$ if and only if
its closure is the union of $\omega$ and some orbits contained in $\maz_{d-1}(k)$.

\medskip
We use the notations introduced in Section \ref{example}.
We choose for $T$ the diagonal torus in $\op{GL}(n)$ and for $B$ the Borel subgroup of upper triangular matrices.
The torus $T$ is naturally isomorphic to the torus $\mathbb{G}_m^n$.
The vector space $V$ is therefore naturally isomorphic to $\mathbb{Q}^n$. 
The Weyl chamber $\mathcal{W}$ corresponding to the chosen Borel subgroup $B$ of $\op{GL}(n)$ is given by
$$\mathcal{W}=\{(a_1,\ldots,a_n)\in V,\quad a_1\geqslant a_2 \geqslant  \ldots \geqslant a_n\}.$$

\begin{propdef}
We denote by $Q$ the set of permutations $g$ of $[\![1,n]\!]$ such that
$$\exists m\in [\![0,n]\!],\quad g_{|g^{-1}([\![1,m]\!])}\text{ is decreasing and }g_{|g^{-1}([\![m+1,n]\!])}\text{ is increasing.}$$
If such an integer $m$ exists, it is unique. We call it the integer associated to $g$ and denote it by $m_g$. 
Also, we denote by $\varepsilon_g$ the function 
\begin{equation*}
\varepsilon_g : [\![1,n]\!] \rw \{+1,-1\}, \quad x\mapsto 
\begin{cases} 
-1 &\text{if }g(x)\in [\![1,m_g]\!]
\\
\\
1 &\text{if }g(x)\in [\![m_g+1,n]\!]
\end{cases}
\end{equation*}
Finally, we denote by $C_g$ the following cone in $V$ : 
 $$C_g:=\{(a_1,\ldots,a_n)\in V,\quad 0\leqslant \varepsilon_g(1)a_{g(1)}\leqslant \cdots\leqslant \varepsilon_g(n)a_{g(n)}\}$$
\end{propdef}

\begin{prop}\label{fangln}
The compactification $\gr_n$ of $\op{GL}(n)$ is log homogeneous and its colored fan consists of the cones $C_g$ and their faces, where $g$ runs over the 
set $Q$.
\end{prop}
\begin{proof}

Let us first prove that the compactification $\gr_n$ is log homogeneous. Let $f$ be an element of $R$. We claim that the complement 
of $\Omega$ in $\maf_{f,n}$ is the union of the coordinate hyperplanes $x_{f(d,1),f(d,2)}=0$, where $d$ runs from $1$ to $n$. Indeed, a point $x$ of 
$\maf_{f,n}$ belongs to $\Omega$ if and only if the point $x'=(b_{f,n}\circ \cdots\circ b_{f,1})(x)$ belongs to $\Omega$. Moreover, we have 
\begin{equation*}
\begin{cases} 
x'_{f(d,1),f(d,2)}=x_{f(1,1),f(1,2)}\ldots x_{f(d,1),f(d,2)} \text{ for all $d$}\in [\![1,n]\!]
\\
x'_{i,j}=0 \text{ if }(i,j)\notin \{(f(1,1),f(1,2)),\ldots, (f(n,1),f(n,2))\}.
\end{cases}
\end{equation*}
By Proposition \ref{0ideal}, the point $x'$ belongs to $\Omega$ if and only if the product 
$$x'_{f(1,1),f(1,2)}\ldots x'_{f(n,1),f(n,2)}=x^n_{f(1,1),f(1,2)}x^{n-1}_{f(2,1),f(2,2)}\ldots x_{f(n,1),f(n,2)}$$ does not vanish, that is, if and only if 
each of the $x_{f(d,1),f(d,2)}$ does not vanish. This proves the claim. In particular, the complement of $\Omega$ in $\gr_n$ is a strict normal crossing divisor.
By Definition \ref{bigdefi}, the variety $\maf_{f,n}$ is a 
smooth toric variety 
for a quotient of the torus $T\times T=L_{f,n}$. From this we see that it is log homogeneous. It is now straightforward to check that
the $P_{f,n}$-variety $\mau_{f,n}$ is log homogeneous. It readily follows that the $G$-variety $\gr_n$ is log homogeneous. 

\medskip
We let $T$ acts on $\gr_n$ on the left. By Section \ref{example}, the closure $\overline{T}$ of $T$ in $\gr_n$ is a toric variety under the action of $T$ and we 
can use the fan of this toric variety to compute the colored fan of $\gr_n$.
We shall now identify some of the cones in the fan of $\overline{T}$. 
Let $g$ be an element of $Q$. We fix an element $f$ of $R$ such that, for each integer $d$ from $1$ to $n$, 
$f(d,1)=(g(d),1)$ if $\varepsilon_g(d)=1$ and $f(d,2)=(g(d),1)$ if $\varepsilon_g(d)=-1$
The variety $\maf_{f,n}$ is an open affine toric subvariety of $\overline{T}$. It corresponds to a cone in the fan of $\overline{T}$, 
namely the cone spanned by the one-parameter subgroups having a limit in $\maf_{f,n}$ at $0$. 
Let 
$$\lambda : \mathbb{G}_m\rw T,\quad t\rw (t^{a_1},\ldots,t^{a_n})$$
be a one-parameter subgroup of $T$. By the formulas in Definition \ref{bigdefi}, we see that the one-parameter subgroup 
$\lambda$ has a limit in $\maf_{f,n}$ at $0$ if and only if 
$$0\leqslant \varepsilon_g(1) a_{g(1)}\leqslant \ldots \leqslant \varepsilon_g(n) a_{g(n)}$$
that is, if and only if $\lambda$ belongs to $C_g$. This proves that, for each element $g$ of $Q$, the cone $C_g$ belongs to the fan
of the toric variety $\overline{T}$.

\medskip
To complete the proof, we show now that the cone $-\mathcal{W}$ is equal to the union of the cones $C_g$, where $g$ runs over the set $Q$. 
Let $g$ be an element of $Q$ and let $(a_1,\ldots,a_n)$ be a point in $C_g$. Let also $i$ be an integer 
between $1$ and $n-1$. If $i<m_g$, then there are two integers $j>j'$ such that $g(j)=i$ and $g(j')=i+1$. These integers satisfy $\varepsilon(j)=-1$ and 
$\varepsilon(j')=-1$. By definition of the cone $C_g$, we have $\varepsilon(j')a_{g(j')}\leqslant \varepsilon(j)a_{g(j)}$, that is, $a_i\leqslant a_{i+1}$.
The same kind of argument prove that $a_i\leqslant a_{i+1}$ for $i= m_g$ and for $i>m_g$.  This proves that the cone $C_g$ is contained in the cone $-\mathcal{W}$. 
We consider now an element $(a_1,\ldots,a_n)$ of $-\mathcal{W}$. By definition it satisfies $a_1\leqslant a_2 \leqslant  \ldots \leqslant a_n$. 
 Let $m$ be an integer such that $a_m\leqslant 0$ and $a_{m+1}\geqslant 0$. The rational numbers $-a_1,\ldots, -a_m$ and $a_{m+1},\ldots, a_n$ are nonnegative. 
 By ordering them in increasing order, we construct an element $g$ of $Q$ such that the point $(a_1,\ldots,a_n)$ belongs to 
 $C_g$.

\end{proof}

\subsection{Fixed points}
In this section, the scheme $S$ is the spectrum of an algebraically closed field $k$ of characteristic not $2$. 
We apply the results obtained in Section \ref{fixed} for some involutions on the log homogeneous compactification $\gr_n$ of $\op{GL}(n)$.
We denote by $J_r$ the antidiagonal square matrix 
of size $r$ with all coefficients equal to one on the antidiagonal.

\medskip

Let $b$ be a nondegenerate symmetric or antisymmetric bilinear form on $k^n$. Via the fixed trivializations of $\mav_1$ and $\mav_2$ we   
obtain nondegenerate symmetric or antisymmetric bilinear forms $b_1$ and $b_2$ on the $k$-vector spaces $\mav_1(k)$ and $\mav_2(k)$.
We equip the direct sum $\mav(k)$ of $\mav_1(k)$ and $\mav_2(k)$ with the nondegenerate symmetric or antisymmetric bilinear form 
$b_1\oplus b_2$. We let $\si$ be the involution of $\gr$ mapping a $n$-dimensional $k$-vector subspace $F$ of $\mav(k)$ to its orthogonal. 
It is an easy exercise to check that 
$$\op{dim}(F^{\perp}\cap \mav_1(k))=\op{dim}(F\cap \mav_2(k)) \text{ and }\op{dim}(F^{\perp}\cap \mav_2(k))=\op{dim}(F\cap \mav_1(k)).$$
By the description of $\maz_d$ given in Section \ref{colored}, we see that the involution $\si$ leaves each of the closed subvarieties $\maz_d$ of $\gr$ invariant.
Therefore it extends to an involution, still denoted $\si$, of each of the varieties $\gr_m$. To prove that we are in the setting of Section \ref{fixed}, it remains 
to observe that there is an involution $\si$ of $\op{GL}(n)$, namely the one associated to $b$, such that 
$$\forall g\in \op{GL}(n)\times \op{GL}(n),\quad \forall x\in \gr_m,\quad \si((g_1,g_2)\cdot x)=\si(g_1)\cdot\si(x)\cdot\si(g_2)^{-1}.$$
As in Section \ref{example}, we denote by $G'$ the neutral component of $G^{\si}$ and by $\gr_n'$ the connected component of $\gr_n^{\si}$ containing $G'$.

\medskip

$\textbf{The odd orthogonal case}.$
We suppose that $n=2r+1$ is odd. We let $b$ be the scalar product with respect to the matrix $J_{2r+1}$.
We have $G':=\op{SO}(2r+1)$.  
We let $T'$ be the intersection of $T$ with $G'$ and $B'$ the intersection of $B$ with $G'$.
The maximal torus $T'$ is naturally isomorphic to the split torus $\mathbb{G}_m^r$ via the following morphism
$$\mathbb{G}_m^r\rw T',\quad (t_1,\ldots, t_r)\mapsto \op{diag}(t_1,\ldots,t_r,1,t_r^{-1},\ldots,t_1^{-1}).$$
The space $V'$ is therefore naturally isomorphic to $\mathbb{Q}^r$. It is contained in $V$ via the following linear map 
$$V'\rw V,\quad (a'_1,\ldots,a'_r,0,-a'_r,\ldots,-a'_1).$$
The Weyl chamber with respect to $B'$ is given by  
$$\mathcal{W}'=\{(a'_1,\ldots,a'_r)\in V,\quad a'_1\geqslant a'_2 \geqslant  \ldots \geqslant a'_r\geqslant 0\}.$$

\begin{prop}\label{oddorth}
The compactification $\gr_n'$ of $G'$ is the wonderful compactification.  
\end{prop}
\begin{proof}
First of all, by Proposition \ref{fixedpoints}, the compactification $\gr'_n$ is log homogeneous. As explained in Section \ref{example}, we use the closure of $T'$ 
in $\gr'_n$ to compute the colored fan of $\gr'_n$. Observe that $T'$ is a subtorus of $T$ and therefore the fan of the toric variety $\overline{T'}$ 
is the trace on $V'$ of the fan
of the toric variety $\overline{T}$. By Proposition \ref{fangln}, the cone 
$$\{(a_1,\ldots,a_{2r+1})\in V, \quad 0\leqslant a_{r+1} \leqslant -a_r \leqslant a_{r+2} \leqslant  \cdots \leqslant -a_1 \leqslant a_{2r+1}\}$$
belongs to the fan of $\overline{T}$. The trace of this cone on $V'$ is $-\mathcal{W}'$, proving that the cone $-\mathcal{W}'$ belongs to the colored fan of $\gr_n'$.
But the only fan in $V'$ with support $-\mathcal{W}'$ containing $-\mathcal{W}'$ is the fan formed by $-\mathcal{W}'$ and its faces. This completes the proof. 
 
\end{proof}

\medskip

$\textbf{The even orthogonal case}.$
We suppose that $n=2r$ is odd. We let $b$ be the scalar product with respect to the matrix $J_{2r}$.
We have $G':=\op{SO}(2r)$.  
We let $T'$ be the intersection of $T$ with $G'$ and $B'$ the intersection of $B$ with $G'$.
The maximal torus $T'$ is naturally isomorphic to the split torus $\mathbb{G}_m^r$ via the following morphism
$$\mathbb{G}_m^r\rw T',\quad (t_1,\ldots, t_r)\mapsto \op{diag}(t_1,\ldots,t_r,t_r^{-1},\ldots,t_1^{-1}).$$
The space $V'$ is therefore naturally isomorphic to $\mathbb{Q}^r$. It is contained in $V$ via the following linear map 
$$V'\rw V,\quad (a'_1,\ldots,a'_r,-a'_r,\ldots,-a'_1).$$
The Weyl chamber with respect to $B'$ is given by  
$$\mathcal{W}'=\{(a'_1,\ldots,a'_r)\in V,\quad a'_1\geqslant a'_2 \geqslant  \ldots \geqslant a'_{r-1}\geqslant |a'_r|\}.$$

\begin{prop}\label{evenorth}
The compactification $\gr_n'$ of $G'$ is log homogeneous and its fan consists of the cones 
$$C_{+}:=\{(a'_1,\ldots,a'_r)\in V,\quad a'_1\leqslant a'_2 \leqslant  \ldots \leqslant a'_{r-1}\leqslant a'_r\leqslant 0\}$$
and 
$$C_{-}:=\{(a'_1,\ldots,a'_r)\in V,\quad a'_1\leqslant a'_2 \leqslant  \ldots \leqslant a'_{r-1}\leqslant -a'_r\leqslant 0\}$$
and their faces.  
\end{prop}
\begin{proof}
The proof is similar to that of Proposition \ref{oddorth}. With the arguments given in this proof it suffices to observe that the trace of the following cone in $V$: 
$$\{(a_1,\ldots,a_{2r})\in V, \quad 0\leqslant -a_r \leqslant a_{r+1}\leqslant \cdots \leqslant -a_1 \leqslant a_{2r}\}$$
on $V'$ is $C_{+}$, the trace of 
$$\{(a_1,\ldots,a_{2r})\in V, \quad 0\leqslant a_r \leqslant -a_{r+1}\leqslant -a_{r-1}\leqslant a_{r+2}\leqslant  \cdots \leqslant -a_1 \leqslant a_{2r}\}$$
is $C_{-}$ and that $-\mathcal{W}'$ is the union of $C_{+}$ and $C_{-}$. 
\end{proof} 
Observe that the Weyl chamber $\mathcal{W}'$ is not smooth with respect to the lattice of one-parameter subgroups of $T'$. 
Therefore the canonical compactification of $G'$ is not smooth, 
and the compactification $\gr'_n$ is a minimal log homogeneous compactification, in the sense that it has a minimal number of closed orbits. 

\medskip

$\textbf{The symplectic case}.$
We suppose that $n=2r$ is even. We let $b$ be the scalar product with respect to the block antidiagonal matrix 
$$\left(
\begin{array}{cc}
0 & -J_r \\
J_r & 0
\end{array}
\right).$$
We have $G':=\op{Sp}(2r)$. 
We let $T'$ be the intersection of $T$ with $G'$ and $B'$ the intersection of $B$ with $G'$.
The maximal torus $T'$ is naturally isomorphic to the split torus $\mathbb{G}_m^r$ via the following morphism
$$\mathbb{G}_m^r\rw T',\quad (t_1,\ldots, t_r)\mapsto \op{diag}(t_1,\ldots,t_r,t_r^{-1},\ldots,t_1^{-1}).$$
The space $V'$ is therefore naturally isomorphic to $\mathbb{Q}^r$. It is contained in $V$ via the following linear map 
$$V'\rw V,\quad (a'_1,\ldots,a'_r,-a'_r,\ldots,-a'_1).$$
The Weyl chamber with respect to $B'$ is given by  
$$\mathcal{W}'=\{(a'_1,\ldots,a'_r)\in V,\quad a'_1\geqslant a'_2 \geqslant  \ldots \geqslant a'_r\geqslant 0\}.$$

\begin{prop}
The compactification $\gr_n'$ of $G'$ is the wonderful compactification.  
\end{prop}
\begin{proof}
The proof is similar to that of Proposition \ref{oddorth}. With the arguments given in this proof it suffices to observe that the trace of the following cone in $V$: 
$$\{(a_1,\ldots,a_{2r})\in V, \quad 0\leqslant -a_r \leqslant a_{r+1} \cdots \leqslant -a_1 \leqslant a_{2r}\}$$
on $V'$ is $-\mathcal{W}'$. 
\end{proof}

\bibliographystyle{plain} 
\bibliography{ma}

\end{document}